\documentclass[openany, twoside]{amsart}
\usepackage[utf8]{inputenc}
\usepackage[english]{babel}
\usepackage{amsfonts, amssymb, amsmath, amsthm}
\usepackage{graphicx}
\usepackage{color}
\usepackage{setspace}
\usepackage{hyperref}
\usepackage{float,comment}
\usepackage{comment}
\numberwithin{equation}{section}
\usepackage[pagewise]{lineno}%\linenumbers

\newtheorem{teor}{Theorem}[section]
\newtheorem{obs}[teor]{Remark}
\newtheorem{lema}[teor]{Lemma}

\newtheorem{cor}[teor]{Corollary}
\newtheorem{prop}[teor]{Proposition}

\theoremstyle{definition}

\newtheorem{Not}[teor]{Notation}

\DeclareMathOperator{\Id}{Id}

\title[Bulk-boundary Bilaplacian eigenvalues]{Bulk-boundary eigenvalues for Bilaplacian problems}

\author[D.\ Buoso]{Davide Buoso }
\thanks{D.\ Buoso, Dipartimento di Scienze e Innovazione Tecnologica, Universit\`a degli Studi del Piemonte Orientale ``A.\ Avogadro'', viale Teresa Michel 11, 15121 Alessandria, Italy}
\email[D.\ Buoso]{davide.buoso@uniupo.it}

\author[C.\ Falcó]{Carles Falcó}
\thanks{C.\ Falcó, Mathematical Institute, University of Oxford, OX2 6GG Oxford, United Kingdom}
\email[C.\ Falcó]{falcoigandia@maths.ox.ac.uk}

\author[M.d.M.\ González]{María del Mar González}
\thanks{M.d.M.\ González, Departamento de Matemáticas, Universidad Autónoma de Madrid and ICMAT, Madrid 28049, Spain}
\email[M.d.M.\ González]{mariamar.gonzalezn@uam.es}

\author[M.\ Miranda]{Manuel Miranda}
\thanks{M.\ Miranda, Institute of Cross-Disciplinary Physics and Complex Systems, IFISC (UIB-CSIC), 07122 Palma de Mallorca, Spain}
\email[M.\ Miranda]{mmiranda@ifisc.uib.es}

\begin{document}
\maketitle

\begin{abstract}
We initiate the study of a bulk-boundary eigenvalue problem for the Bilaplacian with a particular third order boundary condition that arises   from the study of dynamical boundary conditions for the Cahn-Hilliard equation.  First we consider continuity properties under parameter variation (in which the parameter also affects the domain of  definition of the operator). Then we look at the  ball and the annulus geometries (together with the punctured ball), obtaining the eigenvalues as solutions of a precise equation involving special functions. An interesting outcome of our analysis in the annulus case is the presence of a bifurcation from the zero eigenvalue depending on the size of the annulus.
\end{abstract}
\medskip



\section{Introduction}

Let $\Omega$ be a bounded domain in $\mathbb R^n$ for $n \geq 2$ with smooth boundary and exterior normal $\nu$. We consider the energy functional
\begin{equation}\label{energy}
    E_\gamma[u] = \frac{\displaystyle\int_\Omega |\Delta u|^2}{\displaystyle\int_\Omega |u|^2 + \gamma \int_{\partial\Omega} |u|^2},
\end{equation}
for any possible choice of the parameter $\gamma>0$, with the associated energy space
\begin{equation*}
    H^2_\ast (\Omega) = \{ u \in H^2(\Omega): \partial_\nu u = 0 \mbox{ on } \partial \Omega\} \subset H^2(\Omega),
\end{equation*}
 endowed with the standard $H^2$-norm. We remark that the space $H^2_\ast(\Omega)$ is not the only possible choice for the functional \eqref{energy}, however we will not consider any energy space other than that. In this situation, the Euler-Lagrange equation for the associated minimization problem is
 \begin{equation} \label{eq: eigenvalue problem}
    \begin{cases}
     \mathcal L u:=\Delta^2 u = \lambda u & \mbox{in } \Omega, \\ Nu:=\partial_\nu u = 0 & \mbox{on } \partial\Omega, 
    \\ Bu:=-\partial_\nu(\Delta u) = \gamma \lambda u & \mbox{on } \partial\Omega.
    \end{cases}
\end{equation}
The objective of this paper is to study eigenvalues of the bulk-boundary operator $T=(\mathcal L,B)$. Remark that, while the parameter $\gamma$ takes values in the interval $(0,\infty)$, it is possible to formally define the associated problems for $\gamma=0$ and $\gamma=\infty$ (see \eqref{eq:Steklov energy} for the latter). These two extremal cases are particularly interesting because of the disappearance of the bulk-boundary coupling, making them easier to handle.

In the extremal case $\gamma=\infty$, problem \eqref{eq: eigenvalue problem} reduces to
\begin{equation}\label{Laplacian3/2}
    \begin{cases}
    \Delta^2 u = 0 & \mbox{in } \Omega, \\ \partial_\nu u = 0 & \mbox{on } \partial\Omega, 
    \\ -\partial_\nu(\Delta u) =  \lambda u & \mbox{on } \partial\Omega,
    \end{cases}
\end{equation}
which was introduced in the 1960's by \cite{KuttlerSigillito} (see also the references therein), and  several bounds were established. Note that versions of (fourth-order) eigenvalue problems in which the eigenvalue appears in the boundary condition have appeared in the literature \cite{Bucur-Ferrero-Gazzola,Bucur-Gazzola,Buoso,BuosoProvenzano,Ferrero-Gazzola-Weth,Gazzola-Sweers}. These are generally referred to as biharmonic Steklov eigenvalue problems. However, we remark that problem \eqref{eq: eigenvalue problem} has different features: the eigenvalue appears both in the bulk $\Omega$ and on the boundary $\partial \Omega$, coupled by the parameter $\gamma$, which makes it a {\em coupled} fourth-order in $\Omega$, third-order on $\partial\Omega$ eigenvalue problem.

In the particular case that $\Omega$ is a half-space, the operator $B$ is simply the fractional Laplacian $(-\Delta)^{3/2}$ on $\partial\Omega$, up to multiplicative constant (without being exhaustive, we cite \cite{Caffarelli-Silvestre,Chang-Ray,Case}, for instance). We also refer to \cite{LambertiProvenzano}, where the authors study traces of Sobolev functions by means of biharmonic eigenvalue problems that are strictly related to \eqref{eq: eigenvalue problem}.\\

 To the best of our knowledge, bulk-boundary eigenvalue problems first appeared in \cite{vonBelow-Francois} (see also \cite{Francois}), where the authors consider the second order problem
  \begin{equation} \label{eqlaplace bb}
    \begin{cases}
     -\Delta u = \lambda u & \mbox{in } \Omega, \\ 
     \partial_\nu u = \gamma \lambda u & \mbox{on } \partial\Omega,
    \end{cases}
\end{equation}
which can be seen as the analogue of \eqref{eq: eigenvalue problem} for the Laplace operator. They call it a dynamical eigenvalue problem as it can be derived it from the study of parabolic problem with dynamical boundary conditions for the heat equation. Similarly, the motivation behind \eqref{eq: eigenvalue problem} comes from the study of a linear version of the Cahn-Hilliard equation, which is a fourth-order problem. Indeed, consider the following parabolic problem for $v(x,t)$ with dynamical boundary conditions:
\begin{equation}\label{parabolic-problem} 
    \begin{cases}
    \partial_t v=-\Delta^2 v & x\in\Omega,t>0 \\ 
    \partial_\nu v = 0 & \mbox{on } \partial\Omega, \\ 
    \gamma\partial_t v=\partial_\nu(\Delta v)  & \mbox{on } \partial\Omega.
    \end{cases}
\end{equation}
Standard separation of variables $v(x,t)=u(x)T(t)$ then yields the eigenvalue problem \eqref{eq: eigenvalue problem}.

The system \eqref{parabolic-problem} can be understood as a simplified model for the linearization of the  Cahn-Hilliard equation with third-order dynamic boundary conditions, while the introduction of a zero-Neumann condition on $\partial\Omega$ makes the problem more analytically tractable. We only give a brief motivation here and refer to \cite[Section 1.1]{Knopf-Signori} and the references therein for  a summary of the current research in the subject. The Cahn-Hilliard equation is a phase-field model that describes phase separation processes of binary mixtures by a non-linear fourth-order equation,
$$\partial_t v=\Delta \mu,$$
where 
\begin{equation}\label{chemical-potential}\mu=-\Delta v+W'(v)\end{equation}
is the bulk potential in $\Omega$ (here $W$ is a double-well in $\Omega$).

In recent years, several types of dynamic
boundary conditions have been proposed in order to account for the
interactions of the material with the solid wall. In   \cite{Knopf-Lam-Liu-Metzger} the authors introduced a model which can be derived as a gradient flow for a Ginzburg-Landau type energy that contains both a bulk and a boundary contribution (compare to our \eqref{energy}), which gives a third-order dynamic boundary condition $$\partial_t v=-\beta\partial_\nu \mu \quad\text{on }\partial\Omega$$ and
introduces a separate chemical potential $\tilde \mu$ on the boundary  satisfying a relation similar to \eqref{chemical-potential} but with a different double-well on  $\partial\Omega$.  The term $-\partial_\nu\mu$ represents the mass flux at $\partial\Omega$, it is directly influenced by differences of the chemical potentials \begin{equation}\label{equation-L}L\partial_\nu \mu=\beta \tilde \mu-\mu \quad\text{on } \partial\Omega
\end{equation} for a constant $L$ related to the reaction rate.

Singular limits as $L\to 0$ and $L\to \infty$ were also considered in the above mentioned paper \cite{Knopf-Lam-Liu-Metzger}. On the one hand, in the limit $L\to \infty$ (previously studied in \cite{Liu-Wu:Cahn-Hilliard,Garcke-Knopf}) the two potentials are fully decoupled. On the other hand, we  are more interested in the limit $L\to 0$ (introduced in \cite{Goldstein-Miranville-Schimperna}) since it imposes that the two  potentials are  in chemical equilibrium $\mu=\beta\tilde\mu$, related by the parameter $\beta$ (in our setting, $\gamma=1/\beta$). We remark that well-posedness, analytic setup and basic statements for this eigenvalue problem  were established in \cite{Knopf-Liu}, for any $L$. \\

We also observe that the study of \eqref{eq: eigenvalue problem}, at least for the  particular values $\gamma=0$ and $\gamma=\infty$, is partly inspired from an eigenvalue problem in conformal geometry \cite{Gonzalez-Saez}.  In this paper the authors study the (fourth-order) Paneitz operator on locally conformally flat 4-manifolds, and its associated third-order boundary operator (see the papers \cite{Chang-Qing:zeta1,Chang-Qing:zeta2,Case} for the precise definitions). The main term of both operators coincides with $L$ and $B$, respectively, but they include lower order terms that make them conformally covariant (this is, they satisfy a simple intertwining rule under conformal changes of the metric) and contain a lot of geometric/topological information. In particular, the conformally covariant versions of $L,B$ appear in the Gauss-Bonnet formula in four dimensions for manifolds with boundary.

An interesting open question is, thus, to understand the intrinsic geometric meaning of a bulk-boundary coupling.  \\

Here we take an analytic approach and study geometric properties for the model \eqref{eq: eigenvalue problem}. It is easy to see that, for each fixed $\gamma$, there exists a sequence of eigenvalues 
$$0=\lambda_1(\gamma)\leq\lambda_2(\gamma)\leq\ldots.$$  Our first result deals with  the behavior of the eigenvalues and eigenfunctions at the limits $\gamma\to 0$ and $\gamma\to\infty$. Theorem \ref{thm-convergence} shows that they have good convergence properties. This is not completely trivial since, although eigenvalues usually behave continuously under perturbation,  the standard theory does not apply directly when the domain of the operator changes with the parameter (which, in our case, is the non-standard space $L_\gamma^2(\Omega,\partial\Omega)$, defined precisely in \eqref{domain}). Instead, our proof is inspired in the recent paper \cite{GHL} where the authors consider a number of applications of problem \eqref{eqlaplace bb}. \\

Next, we obtain expressions  for the eigenvalues and eigenfunctions in the ball (Theorem \ref{thm:general} below). More precisely, under the standard spherical harmonic decomposition, for each fixed projection $\ell=0,1,\ldots$, eigenfunctions  can be written in terms of special functions, while  eigenvalues are obtained as the roots of the equation \eqref{gammaeq}. Since this is a nontrival expression, we have given some numerical approximations, as well as looking at the limits $\gamma\to 0$
and $\gamma\to\infty$ explicitly. In any case, we show that the fundamental mode (the first non-zero eigenvalue) in the ball case occurs at the projection $\ell=1$ (see Theorem \ref{thm:fundamental-mode}), as one would expect for the ball geometry. 

Then we move to the annular geometry (and the punctured ball as a limiting case), obtaining similar formulas for the eigenvalues and eigenfunctions after projection over spherical harmonics. However, here we find a surprising behavior: indeed, a bifurcation from the zero eigenvalue may happen when varying the size of the annulus, which implies that there is a non-trivial radially symmetric eigenfunction (corresponding to $\ell=0$) whose associated eigenvalue is the fundamental mode. This picture is more easily understood by looking at Figures  \ref{n2_annulus} and \ref{n4_annulus} from Section \ref{section:annulus}. We remark that a similar phenomena can be  observed in the calculation of the lowest non-zero Steklov eigenvalue in a 2-dimensional annulus from \cite{Fraser-Schoen} or the lowest non-zero eigenvalue for the third-order boundary operator associated to the fourth-order Paneitz operator on a locally conformally flat 4-manifold \cite{Gonzalez-Saez}.

Regarding the punctured ball, we observe that the behavior is precisely that of the full ball in any dimension, meaning in particular that the eigenfunctions can be analytically continued up to the origin.  Notice that this phenomenon is different from what is observed for other Bilaplacian eigenvalue problems, such as the Dirichlet case or the buckling problem, where  in the low dimension regime ($n=2,3$) the radial eigenfunctions are the only ones that separate the ball from the punctured ball, while in the high dimension regime ($n\ge4$) the two problems coincide (see e.g., \cite{buosoparini, coffman1, coffman2}).\\

Our last observation is that problem \eqref{eq: eigenvalue problem} is closely related both to a Steklov-type problem and to a Neumann-type one, standing somewhat in the middle. This has deep consequences at the geometrical level, since it qualitatively changes the behavior of any domain optimization problem for the fundamental mode, which is a challenging problem that is outside the scope of this paper. In any case,  for the two extremal cases (that, we recall, are decoupled), optimization has already been considered: for $\gamma=0$, it is an immediate consequence of the fact that the problem reduces to the Neumann Laplacian squared (see equation \eqref{eq:neumann-squared} and the discussion right below), while the case $\gamma=\infty$ was studied in \cite{Xia-Wang}. In both the extremal cases the ball turns out to be the maximizer for the fundamental tone. 

On the other hand, for other values of $\gamma$ it is not even clear if the ball is a maximizer. Nevertheless, the arguments used in \cite{Buoso} (see also the references therein) suggest that the ball is at least a critical point under volume or perimeter constraint (cf. \cite{buosokennedy}). We also remark that, while continuity of the eigenvalues is expected to hold for smooth perturbations of the domain, the eigenvalues of problem \eqref{eq: eigenvalue problem} may not behave well if the perturbation is not smooth. We refer the interested reader to \cite{courant-hilbert, hale} for classical results on the Laplacian, as well as the more recent \cite{bll} for other related results regarding the spectral stability of nonnegative operators and to \cite{bogosel} for the Steklov problem.\\

 The paper will be structured as follows: first, in Section \ref{section:setup} we give all the functional analytic preliminaries. In Section \ref{section:limits} we give our first results and prove convergence as $\gamma \to 0$ and $\gamma\to \infty$. Finally, in Sections \ref{section:ball} and \ref{section:annulus} we consider the ball and annulus domains, respectively.

\section{Functional analytic setup}\label{section:setup}
Let $n\ge2$, let $\Omega$ be a bounded domain in $\mathbb R^n$, and let $\nu$ denote the exterior unit normal vector to $\partial \Omega$. For the sake of simplicity we will always consider $\partial\Omega\in C^2$, even though in principle it may be possible to consider the problems we study under a less smooth environment. Integrals will be over $\Omega$ or $\partial \Omega$, with their respective measures that will not be written explicitly in this exposition. Also, we will set $s = \frac{n-2}{2}$.
\\

We denote by $H^2(\Omega)$ the set of (real-valued) functions in $L^2(\Omega)$ whose first and second derivatives are all in $L^2(\Omega)$. We recall that, in this case, it makes sense to consider the trace operators $\Gamma_0,\Gamma_1:H^2(\Omega)\mapsto L^2(\partial\Omega)$ which are the extensions of the classical restriction operators
$$
\Gamma_0(u)=u|_{\partial\Omega},\ \ \Gamma_1(u)=(\nabla u)|_{\partial\Omega}\cdot\nu,\ \ \forall u\in H^2(\Omega)\cup C^{\infty}(\overline{\Omega}).
$$
Without loss of generality, we will denote by $u,\partial_\nu u$ the traces of an $H^2$-function. We refer to \cite{burenkov} for properties of Sobolev spaces and their traces.

We can therefore introduce the following space
\begin{equation}
\label{h2star}
    H^2_\ast (\Omega) = \{ u \in H^2(\Omega): \partial_\nu u = 0 \mbox{ on } \partial \Omega\} \subset H^2(\Omega),
\end{equation}
%endowed with the $H^2$ norm, namely:
%\begin{equation*}
    %\lVert u \lVert_{H^2}^2 = \int_\Omega [u^2 + |\nabla u|^2 + (\Delta u)^2],
%\end{equation*}
which is a Hilbert subspace of $H^2(\Omega)$ when endowed with the standard scalar product
\begin{equation*}
    \langle u,v \rangle_{H^2_\ast} = \int_\Omega [uv + (\Delta u)(\Delta v)],
\end{equation*}
producing the norm
\begin{equation} \label{eq:norm}
    \lVert u \lVert_{H^2_\ast}^2 = \int_\Omega [u^2 + (\Delta u)^2].
\end{equation}
%Note that the condition $\partial_\nu u = 0$ is well defined in the trace sense when $u \in H^2$, as it implies $\nabla u \in H^1(\Omega)$.

For any $\gamma \in (0,\infty)$, we also define the space 
\begin{equation}\label{domain}
L_\gamma^2(\Omega,\partial\Omega):=\{(u,u|_{\partial\Omega}): u\in H^2(\Omega)\},
\end{equation}
which turns out to be a Hilbert space when endowed with the scalar product
\begin{equation*}
\langle u,v\rangle_{L_\gamma^2(\Omega,\partial\Omega)}
:=\int_\Omega uv + \gamma \int_{\partial\Omega} uv.
\end{equation*}
The Trace inequality immediately implies that $H^2_*(\Omega)$ is compactly embedded in $L_\gamma^2(\Omega,\partial\Omega)$,
%\begin{equation*}
%    \lVert u \lVert_{L^2(\partial\Omega)} \leq C \lVert u \lVert_{H^1(\Omega)} \leq C \lVert u \lVert_{H^2(\Omega)}.
%\end{equation*}
%
leading to the Hilbert triple:
\begin{equation*}
    H^2_\ast(\Omega)  \hookrightarrow L^2_\gamma (\Omega,\partial\Omega) \hookrightarrow H^2_\ast(\Omega)^*,
\end{equation*}
where by $H^2_\ast(\Omega)^*$ we denote the dual of $H^2_\ast(\Omega)$.

Given $f \in L^2(\Omega)$ and $h \in L^2(\partial\Omega)$, we consider the boundary value problem
\begin{equation}\label{eq}
\begin{cases}
Lu:=\Delta^2 u=f &  \text{in }\Omega,\\
Nu:=\partial_\nu u=0&  \text{on }\partial\Omega,\\
Bu:=-\partial_\nu(\Delta u)=h&  \text{on }\partial\Omega,
\end{cases}
\end{equation}
whose associated  bilinear functional is given by
\begin{equation*}\label{bilinear}
\mathcal Q(u,v)=\int_\Omega (\Delta u)(\Delta v).
\end{equation*}
%with associated energy
%\begin{equation*}
%\mathcal Q(u)=\int_\Omega (\Delta u)^2.
%\end{equation*}
In this notation, we can define weak solutions of \eqref{eq}
 as those functions $u \in H^2_\ast(\Omega)$ satisfying
\begin{equation}\label{eq:variationeq}
    \mathcal{Q}(u,v) = \int_\Omega fv + \gamma\int_{\partial\Omega} hv \quad \forall v \in H^2_\ast(\Omega).
\end{equation}

It is important to remark that the boundary conditions $N,B$ satisfy the so-called \emph{complementing condition} of \cite{ADN1,ADN2}. This is, for any tangent vector at the boundary $\tau(x)$ and for any orthogonal vector $v(x)$, the boundary operators $N(\tau + tv)$ and $B(\tau + tv)$ are linearly independent modulo the polynomial $(t+i|\tau|)^2$. For our case, it means that the polynomials $t$ and $t^3+t$ are linearly independent modulo $(t+i)^2$. Satisfying this complementing condition is crucial in order to have solvability of polyharmonic problems, see also \cite[Chaper 2]{GGS}.

Existence of solutions of problem \eqref{eq} follows from a standard Fredholm argument thanks to the energy estimate $
   \lVert u \lVert_{H^2_\ast}^2 \leq \mathcal{Q}(u,u) + \lVert u \lVert_{L^2(\Omega)}^2.$
Thus equation \eqref{eq} is well posed, has good Fredholm properties, and  if the compatibility condition
\begin{equation*}
\int_\Omega f+\gamma\int_{\partial \Omega}h=0
\end{equation*}
holds, then there exists a one-dimensional set of solutions to  \eqref{eq}. Two such solutions only differ by a constant.

Regularity for solutions of polyharmonic problems has been well studied. In particular, given $k\geq 4$, for a-priori estimates in Sobolev spaces  we recall (\cite[Theorem 2.20]{GGS}) that  
\begin{equation}\label{estimate1}
   \lVert u \lVert_{W^{k,p}(\Omega)} \leq C \left(\lVert f \lVert_{W^{k-4,p}(\Omega)}+ \lVert h \lVert_{W^{k-3-\frac{1}{p},p}(\partial\Omega)} + \lVert u \lVert_{L^1(\Omega)}\right),
\end{equation}
while for H\"older regularity  we have (\cite[Theorem 2.19]{GGS})
\begin{equation}\label{estimate2}
\|u\|_{C^{k, \gamma}(\bar{\Omega})} \leq C\left(\left\|f\right\|_{C^{k-4, \gamma}(\bar{\Omega})}+\left\|h \right\|_{C^{k-3, \gamma}(\partial \Omega)} + \lVert u \lVert_{L^1(\Omega)}\right).
\end{equation}
Note that both estimates can be localized.\\

In the notation of \eqref{eq}, we define the operator  $T$  over the space $L_\gamma^2(\Omega,\partial\Omega)$, given by the pair $(L,B)$ subject to the Neumann condition $N=0$. We are interested in the corresponding eigenvalue problem, namely equation \eqref{eq: eigenvalue problem}.
It is standard to show that $T$ is a non-negative, densely defined, self-adjoint operator with compact resolvent, therefore its spectrum consists of a non-decreasing sequence of eigenvalues of finite multiplicities
$$
0=\lambda_1(\gamma,\Omega)\le \lambda_2(\gamma,\Omega)\le\lambda_3(\gamma,\Omega)\le\dots\le\lambda_k(\gamma,\Omega)\le\dots
$$
where each individual eigenvalue is repeated accordingly. In addition, the eigenfunctions $\{u_k\}_{k=1}^\infty$ can be normalized to form an orthonormal basis for $H^2_\ast(\Omega)$. We observe that the first eigenvalue is always zero with associated eigenfunction the constant function on $\Omega$. Moreover, bootstrapping the estimates \eqref{estimate1}-\eqref{estimate2} yields that all eigenfunctions are smooth. Since $\Omega$ is assumed to be connected, necessarily $\lambda_2(\gamma,\Omega)>0$. Notice that the eigenvalues always depend both on $\gamma$ and on $\Omega$; however, we will drop either argument when clear from the context, in order to ease the notation.

We also recall that, by the Courant min-max principle, the eigenvalues can be variationally characterized as
\begin{equation}
    \label{min-max}
\lambda_k(\gamma,\Omega)=\min_{\substack{V\subseteq H^2_\ast(\Omega),\\ \dim V=k}}\max_{u\in V\setminus\{0\}}E_\gamma[u],
\end{equation}
where $E_\gamma$ is the Rayleigh quotient associated with $T$:
\begin{equation} \label{eq:energy}
    E_\gamma[u] := \frac{\mathcal{Q}(u)}{\lVert u \lVert_{L_\gamma^2(\Omega,\partial\Omega)}} = \frac{\int_\Omega |\Delta u|^2}{\int_\Omega |u|^2 + \gamma \int_{\partial\Omega} |u|^2}.
\end{equation}
In particular, the first non-zero eigenvalue (the so-called \emph{fundamental tone}) can be written as
$$\lambda_2(\gamma,\Omega)=\min_{u\in\mathcal A_\gamma,u\not\equiv 0}E_\gamma[u],$$
where $\mathcal A_\gamma$ is the set of functions
\begin{equation*}
\mathcal A_\gamma:=\left\{u\in H^2_*(\Omega)\,:\,\int_\Omega u + \gamma\int_{\partial\Omega}u=0  \right\}.
\end{equation*}

We note that this problem has a scaling property that affects also the $\gamma$ factor, due to the scale difference between the interior measure and the boundary surface area. More precisely for $R>0$, we call $R\Omega = \{ x \in \mathbb{R}^n : x/R \in \Omega\}$. Then, if we denote $\lambda(\gamma,\Omega)$ the eigenvalue in \eqref{eq: eigenvalue problem} for a fixed $\gamma$ and a domain $\Omega$, we have that
\begin{equation*}
    \lambda(\gamma,\Omega) = R^4 \lambda(R\gamma, R\Omega).
\end{equation*}

Next, although the previous discussion considered only the situation $\gamma\in(0,\infty)$, the two limiting cases $\gamma=0$ and $\gamma=\infty$ are actually very interesting. If we set $\gamma=0$ in \eqref{eq: eigenvalue problem} we obtain
 \begin{equation} \label{eq:neumann-squared}
    \begin{cases}
     \Delta^2 u = \lambda u & \mbox{in } \Omega, \\ \partial_\nu u = \partial_\nu(\Delta u) = 0 & \mbox{on } \partial\Omega,
    \end{cases}
\end{equation}
which can be associated with the Poisson-type problem 
\begin{equation*} 
    \begin{cases}
     \Delta^2 u = f & \mbox{in } \Omega, \\ \partial_\nu u = \partial_\nu(\Delta u) = 0 & \mbox{on } \partial\Omega,
    \end{cases}
\end{equation*}
for any $f\in L^2(\Omega)$ (see \cite[Section 1.1.3]{GGS} for a discussion of the Navier problem for the Bilaplacian). The latter, if we set $v=-\Delta u$, can be decomposed into the system of differential equations
\begin{equation*} 
    \begin{cases}
     -\Delta v = f & \mbox{in } \Omega, \\ \partial_\nu v =0 & \mbox{on } \partial\Omega,
    \end{cases}
		\qquad\text{and}\qquad
		\begin{cases}
     -\Delta u = v & \mbox{in } \Omega, \\ \partial_\nu u =0 & \mbox{on } \partial\Omega.
    \end{cases}
\end{equation*}
Since $\Omega$ is smooth, this immediately implies that the associated operator is the classical Neumann Laplacian squared, that is, the eigenvalues of \eqref{eq:neumann-squared} are the squared of the eigenvalues of
 \begin{equation*} %\label{eq:neumann-squared}
    \begin{cases}
     -\Delta u = \lambda u & \mbox{in } \Omega, \\ \partial_\nu u = 0 & \mbox{on } \partial\Omega.
    \end{cases}
\end{equation*}
We refer to \cite{adolfsson, GGS} for more details on this identification. Also in this case, the associated operator is self-adjoint with compact resolvent, resulting in a non-decreasing sequence of eigenvalues of finite multiplicities, diverging to plus infinity, and the corresponding eigenfunctions can be normalized to form an orthonormal basis of $H^2_\ast(\Omega)$. We will list them as $\lambda_k(0,\Omega)$.\\

In addition, when $\gamma=\infty$, if we consider instead the Rayleigh quotient
\begin{equation} \label{eq:Steklov energy}
\tilde E_\gamma [u] = \frac{\int_\Omega (\Delta u)^2}{\frac{1}{\gamma}\int_\Omega u^2 +  \int_{\partial \Omega} u^2} = \gamma E_\gamma [u],
\end{equation}
it is easy to see that this case is associated with the Steklov-type eigenvalue problem
 \begin{equation} \label{eq:ks-eig}
    \begin{cases}
     \Delta^2 u = 0 & \mbox{in } \Omega, \\ \partial_\nu u = 0 & \mbox{on } \partial\Omega,\\
		-\partial_\nu(\Delta u) = \lambda u & \mbox{on } \partial\Omega.
    \end{cases}
\end{equation}
Problem \eqref{eq:ks-eig} was first introduced in \cite{KuttlerSigillito}, and has recently gained attention for its relations with the other biharmonic problems, as well as related questions (see e.g., \cite{buosokennedy, LambertiProvenzano, Liu}). Problem \eqref{eq:ks-eig} admits as well a non-decreasing sequence of eigenvalues of finite multiplicities, diverging to plus infinity, and the corresponding eigenfunctions can be normalized to form an orthonormal basis of $H^2_\ast(\Omega)$. We will list them as $\lambda_k(\infty,\Omega)$.

In this case, an application of Fichera's duality principle (see \cite{fichera}) allows to provide an additional characterization of the first non-trivial eigenvalue of \eqref{eq:ks-eig}. Let us set 
\begin{equation}\label{ficheras}
\tau(\Omega)=\inf_{0\neq w\in {\mathcal B}_\infty(\Omega)} 
\frac{\int_{\partial\Omega} (\partial_\nu w)^2 }{\int_\Omega w^2}
=\inf_{0\neq w\in \overline{{\mathcal B}}_\infty(\Omega)} 
\frac{\int_{\partial\Omega} (\partial_\nu w)^2 }{\int_\Omega w^2},
\end{equation}
where
\begin{equation*}
{\mathcal B}_\infty(\Omega):=\left\{w\in C^{2}(\overline{\Omega}) \, : \,\Delta w=0 \text{ in }\Omega,\, \int_\Omega w=0\right\},
\end{equation*}
and $\overline{{\mathcal B}}_\infty(\Omega)$ is the closure of ${\mathcal B}_\infty(\Omega)$ endowed with the norm
$$
\|w\|_{\mathcal B_\infty(\Omega)}=\|\partial_\nu w\|_{L^2(\partial \Omega)}.
$$

\begin{teor}
Suppose $\Omega\subseteq\mathbb R^n$ is a bounded Lipschitz open set. Then  $\overline{{\mathcal B}}_\infty(\Omega)\hookrightarrow L^2(\Omega)$ compactly, and problem \eqref{ficheras} admits a minimizer $w\in\overline{\mathcal B}_{\infty}(\Omega)$. If in addition $\partial\Omega\in C^2$, then
\begin{equation}\label{fks}
\tau(\Omega)=\lambda_2(\infty,\Omega),
\end{equation}
and moreover the minimizers of $\tau(\Omega)$ are the Laplacians of the eigenfunctions associated with $\lambda_2(\infty,\Omega)$.
\end{teor}

We remark that the identity \eqref{fks} has already appeared in \cite{KuttlerSigillito}, although the authors do not provide an explicit proof and refer to \cite{fichera}. We add a proof of this fact here for the reader's convenience (cf.\ \cite[Theorem 3.23]{GGS}).

\begin{proof}
We start showing that the embedding $\overline{{\mathcal B}}_\infty(\Omega)\hookrightarrow L^2(\Omega)$ is compact when $\Omega$ is bounded and with Lipschitz boundary. It is  trivial to observe that $\overline{{\mathcal B}}_\infty(\Omega)\subseteq H^2(\Omega)$, since
$$
\|w\|_{\mathcal B_\infty(\Omega)}=\|\partial_\nu w\|_{L^2(\partial \Omega)}\le C \|w\|_{H^2(\Omega)},
$$
for some positive constant $C$. Now, combining \cite[Theorem 2]{jerken} with \cite[Proposition 2.3]{mitrea} we have that
$$
\|w\|_{H^1(\Omega)}\le C_1 \|\nabla w\|_{L^2(\Omega)}\le C_2\|w\|_{\mathcal B_\infty(\Omega)},
$$
for some positive constants $C_1,C_2$. This means that the embedding $\overline{{\mathcal B}}_\infty(\Omega)\hookrightarrow H^1(\Omega)$ is continuous, hence the compactness of $\overline{{\mathcal B}}_\infty(\Omega)\hookrightarrow L^2(\Omega)$.

Now we prove that the minimization problem \eqref{ficheras} admits a minimizer $w$. Consider a minimizing sequence $(w_n)_n$ such that $\|w_n\|_{{\mathcal B}_\infty(\Omega)}=1$. Up to a subsequence, there exists $w\in \overline{{\mathcal B}}_\infty(\Omega)$ such that $w_n\rightharpoonup w$, hence $w_n\to w$ in $L^2(\Omega)$. In particular
$$
\|w\|_{L^2(\Omega)}^{-2}=\lim_n \|w_n\|_{L^2(\Omega)}^{-2}=\tau,
$$
while the weak lower semicontinuity of the norm implies that
$$
\|\partial_\nu w\|_{L^2(\partial\Omega)}=
\|w\|_{{\mathcal B}_\infty(\Omega)}\le\liminf_n \|w\|_{{{\mathcal B}_\infty(\Omega)}}=1,
$$
namely
$$
\frac{\int_{\partial\Omega} (\partial_\nu w)^2 }{\int_\Omega w^2}\le\tau,
$$
proving that $w$ is a minimizer.

Now suppose that $\partial\Omega\in C^2$, and take a minimizer $w$ of $\tau(\Omega)$. Then, the Euler-Lagrange equation associated with the minimization problem reads
$$
\int_{\partial\Omega}\partial_\nu w\,\partial_\nu v=\tau(\Omega)\int_\Omega wv,\qquad\forall v\in\overline{{\mathcal B}}_\infty(\Omega).
$$
We set
$$
\begin{cases}
-\Delta u=w & \text{in}\ \Omega,\\
\partial_\nu u=0 & \text{on}\ \partial\Omega.
\end{cases}
$$
Note that there is a unique solution $u\in {\mathcal B}_\infty(\Omega)$ since $w\in \overline{{\mathcal B}}_\infty(\Omega)$ and $\Omega$ is smooth, and in particular $u\in H^2_\ast(\Omega)$. Hence
$$
\int_\Omega w v=-\int_\Omega v\Delta u =\int_\Omega \nabla u\cdot\nabla v=\int_{\partial\Omega}u\partial_\nu v,
$$
from which we get that
\begin{equation}\label{transf}
-\int_{\partial\Omega}\partial_\nu \Delta u\,\partial_\nu v=\tau(\Omega)\int_{\partial\Omega}u\partial_\nu v.
\end{equation}
Now, since $v$ varies in $\overline{{\mathcal B}}_\infty(\Omega)$, we get that $\partial_\nu v$ (on $\partial\Omega$) varies in $H^{3/2}(\partial\Omega)$, so that \eqref{transf} implies that $u$ satisfies
$$
-\partial_\nu \Delta u=\tau(\Omega)u\quad\text{on}\ \partial\Omega.
$$
Multiplying both sides by $u$ we obtain
$$
\tau(\Omega)\int_{\partial\Omega}u^2
=-\int_{\partial\Omega}u\partial_\nu \Delta u
=-\int_{\Omega}\nabla u\cdot\nabla \Delta u
=\int_\Omega (\Delta u)^2,
$$
meaning in particular that
$$
\tau(\Omega)=\frac{\int_\Omega (\Delta u)^2}{\int_\Omega u^2}
\ge \min_{u\in H^2_\ast(\Omega), \int_\Omega u=0}\frac{\int_\Omega (\Delta u)^2}{\int_\Omega u^2}=\lambda_2(\infty,\Omega),
$$
where the last equality is due to the fact that the additional condition $\int_\Omega u=0$ in the minimax characterization is quotienting away the constants (associated with $\lambda_1(\infty,\Omega)=0$). This proves that $\tau(\Omega)\ge\lambda_2(\infty,\Omega)$.

To prove the other verse of the inequality, we now consider $u\in H^2_\ast(\Omega)$ an eigenfunction associated with $\lambda_2(\infty,\Omega)$. By the regularity estimate \eqref{estimate1} it is easy to see that $w=-\Delta u$ belongs to $H^2(\Omega)$, therefore $u$ satisfies a.e.\ the equation
$$
-\partial_\nu \Delta u=\lambda_2(\infty,\Omega)u\quad\text{on}\ \partial\Omega.
$$
It is now possible to multiply both sides by $\partial_\nu v$ for $v\in {\mathcal B}_\infty(\Omega)$, and repeating the same computations as before we get
$$
\int_{\partial\Omega}\partial_\nu w\,\partial_\nu v=\lambda_2(\infty,\Omega)\int_\Omega wv,
$$
which shows that $\lambda_2(\infty,\Omega)\ge \tau(\Omega)$,
as desired.
\end{proof}

\section{Convergence to the limiting cases}\label{section:limits}

In this section we analyze the limiting behavior of the eigenvalues $\lambda_k(\gamma, \Omega)$ of problem \eqref{eq: eigenvalue problem} as the parameter $\gamma$ tends either to zero or to plus infinity. We will achieve this result by adapting the arguments of \cite[Theorem 5]{GHL}. 

We start with the following:

\begin{teor}
For any smooth and bounded domain $\Omega$, and for any $k\in\mathbb N$, the map
$$
\gamma\mapsto \lambda_k(\gamma)
$$
is non-increasing in $[0,+\infty)$, while the map
$$
\gamma\mapsto \gamma\lambda_k(\gamma)
$$
is non-decreasing in $[0,+\infty)$.
\end{teor}

\begin{proof}
The proof follows from the variational characterization of the eigenvalues: given $\gamma_1<\gamma_2$ we immediately see that
$$
E_0[u]\ge E_{\gamma_1}[u]\ge E_{\gamma_2}[u],
$$
for any $u\in H^2_\ast(\Omega)$, so that applying the min-max formula \eqref{min-max} we obtain as a consequence that
$$
\lambda_k(\gamma_1)\ge \lambda_k(\gamma_2),
$$
for any $k\in\mathbb N$. On the other hand, 
$$
\tilde E_{\gamma_1}[u]\le \tilde E_{\gamma_2}[u],
$$
for any $u\in H^2_\ast(\Omega)$, so that applying once more the min-max principle we obtain as a consequence that
$$
\gamma_1 \lambda_k(\gamma_1)\le \gamma_2\lambda_k(\gamma_2),
$$
for any $k\in\mathbb N$.
\end{proof}

\begin{cor}\label{cor:monotonicity}
For any $k$ and for any $\gamma\in(0,+\infty)$ the following hold:
$$
\lambda_k(0)\ge\lambda_k(\gamma)\ \text{ and }\ \gamma\lambda_k(\gamma)\le\lambda_k(\infty).
$$
\end{cor}

We are now ready to prove the convergence as $\gamma\to 0$.
\begin{teor}\label{thm-convergence}
For any smooth and bounded domain $\Omega$, and for any $k\in\mathbb N$, the function 
$$
\gamma\mapsto \lambda_k(\gamma)
$$
is continuous in $[0,+\infty)$ and we have that
$$
\lambda_k(\gamma,\Omega)\to\lambda_k(0,\Omega), \text{ as }\gamma\to0.
$$
More precisely, given an eigenpair $(\lambda_k(\gamma,\Omega),u_k(\gamma))$ there exists an eigenfunction $u_k(0)$ of problem \eqref{eq:neumann-squared} associated to the $k$-th eigenvalue $\lambda_k(0,\Omega)$ such that
$$
\lambda_k(\gamma,\Omega)\to\lambda_k(0,\Omega) \ \text{ and }\ u_k(\gamma)\xrightarrow{H^2_\ast} u_k(0), \text{ as }\gamma\to0.
$$
On the other hand, there exists an eigenfunction $u_k(\infty)$ of problem \eqref{eq:ks-eig} associated to the $k$-th eigenvalue $\lambda_k(\infty,\Omega)$ such that
$$
\gamma\lambda_k(\gamma,\Omega)\to\lambda_k(\infty,\Omega) \quad \text{ and }\quad \gamma^{\frac 1 2} u_k(\gamma)\xrightarrow{H^2_\ast} u_k(\infty), \quad\text{as }\gamma\to\infty.
$$
\end{teor}

\begin{proof}
We start by considering the case $\gamma\to0$ and we observe that, since $\lambda_k(\gamma)$ is non-increasing in $\gamma$, then the limit for $\gamma\to0$ must exist, and we pose
$$
\bar{\lambda}_k=\lim_{\gamma\to0}\lambda_k(\gamma).
$$
Now we consider the eigenpair $(\lambda_k(\gamma),u_k(\gamma))$, where the eigenfunction $u_k(\gamma)$ is normalized in $L^2_\gamma(\Omega,\partial\Omega)$. We have that
$$
\int_\Omega |u_k(\gamma)|^2\le\|u_k(\gamma)\|^2_{L^2_\gamma(\Omega,\partial\Omega)}\le 1,\quad \text{and}\quad \int_\Omega(\Delta u_k(\gamma))^2=\lambda_k(\gamma),
$$
which tells us that the set $\{u_k(\gamma)\}_{\gamma>0}$ is bounded in $H^2_\ast(\Omega)$. We can then extract a subsequence of such eigenfunctions which is weakly convergent in $H^2_\ast(\Omega)$ as $\gamma\to0$, and we call the limit $\bar u_k$. Notice that also the sequence of associated eigenvalues will converge to $\bar \lambda_k$. It holds  that $(\bar\lambda_k,\bar u_k)$ is a solution to the following problem
$$
\int_{\Omega}\Delta \bar u_k \Delta v=\bar\lambda_k\int_{\Omega}\bar u_k v,\ \forall v\in H^2_\ast(\Omega),
$$
that is, it is an eigenpair for problem \eqref{eq:neumann-squared}.

Now we would like to prove that $(\bar\lambda_k,\bar u_k)=(\lambda_k(0),u_k(0))$, and we argue by induction on the index $k$ of the eigenpair. We remark that the multiplicity of $\lambda_k(\gamma)$ does not interfere with the argument we use, since for multiple eigenvalues we can always choose an associated eigenfunction from the corresponding eigenspace which does not overlap with the other eigenfunctions. Since the case $k=1$ is obvious, we assume  convergence for all the eigenpairs up to the index $k-1$. We set
\begin{equation}\label{decomposition}
u_k(0)=U(\gamma)+\sum_{j=1}^{k-1}\langle u_k(0),u_j(\gamma) \rangle_{L^2_\gamma} \,u_j(\gamma),
\end{equation}
for some function $U$.
This is clearly an orthogonal decomposition, and we furthermore decompose
\begin{equation}\label{orth}
\langle u_k(0),u_j(\gamma) \rangle_{L^2_\gamma}=
\langle u_k(0),u_j(0) \rangle_{L^2_\gamma}+
\langle u_k(0),u_j(\gamma)-u_j(0) \rangle_{L^2_\gamma}.
\end{equation}
For the first term in the right-hand side of \eqref{orth} we see that
$$
\langle u_k(0),u_j(0) \rangle_{L^2_\gamma}=\int_\Omega u_k(0) u_j(0) +\gamma\int_{\partial\Omega}u_k(0)u_j(0),
$$
where the first integral vanishes by hypothesis, and the second goes to zero as $\gamma\to0$.  For the second term in the right-hand side of \eqref{orth} we have instead that
$$
\langle u_k(0),u_j(\gamma)-u_j(0) \rangle_{L^2_\gamma}=\int_\Omega u_k(0) [u_j(\gamma)-u_j(0)] +\gamma\int_{\partial\Omega}u_k(0)[u_j(\gamma)-u_j(0)].
$$
Here the inductive hypothesis that $u_j(\gamma)\xrightarrow{H^2_\ast}u_j(0)$ combined with classical Sobolev embeddings and the Trace Theorem allows again to conclude that this term goes to zero as $\gamma\to0$. Therefore
$$
\text{for all}\quad 1\le j\le k-1,\quad \langle u_k(0),u_j(\gamma) \rangle_{L^2_\gamma}\to0\ \text{ as }\ \gamma\to0.
$$
Now we have, taking into account the decomposition \eqref{decomposition},
\begin{equation}\label{longexpr}\begin{split}
\lambda_k(0) &= \mathcal Q(u_k(0),u_k(0))\\
 &=\int_\Omega |\Delta U(\gamma)|^2+\sum_{j=1}^{k-1} \langle u_k(0),u_j(\gamma) \rangle_{L^2_\gamma}^2 \int_\Omega |\Delta u_j(\gamma)|^2
+2\sum_{j=1}^{k-1} \langle u_k(0),u_j(\gamma) \rangle_{L^2_\gamma} \int_\Omega (\Delta u_j(\gamma))(\Delta U(\gamma))\\
&\ge \int_\Omega |\Delta U(\gamma)|^2+\sum_{j=1}^{k-1} \langle u_k(0),u_j(\gamma) \rangle_{L^2_\gamma}^2 \int_\Omega |\Delta u_j(\gamma)|^2\\
&-2\sum_{j=1}^{k-1} \left|\langle u_k(0),u_j(\gamma) \rangle_{L^2_\gamma}\right| \left(\int_\Omega (\Delta u_j(\gamma))^2\right)^{\frac 1 2}\left(\int_\Omega(\Delta U(\gamma))^2\right)^{\frac 1 2}.
\end{split}\end{equation}
We observe that the last two terms above converge to zero as $\gamma\to0$, while the Courant min-max principle implies that
$$
\int_\Omega |\Delta U(\gamma)|^2\ge \lambda_k(\gamma)\left(\int_\Omega U(\gamma)^2+\gamma\int_{\partial\Omega} U(\gamma)^2\right).
$$
Here again the second term goes to zero, while
\begin{equation}\label{lastchain}\begin{split}
\int_\Omega U(\gamma)^2&=
\int_\Omega u_k(0)^2\\
&-2\sum_{j=1}^{k-1}\langle u_k(0),u_j(\gamma) \rangle_{L^2_\gamma} \int_\Omega u_j(\gamma) u_k(0)\\
&+\sum_{j,l=1}^{k-1}\langle u_k(0),u_j(\gamma) \rangle_{L^2_\gamma}\langle u_k(0),u_l(\gamma) \rangle_{L^2_\gamma} \int_\Omega u_j(\gamma)u_l(\gamma).
\end{split}\end{equation}
Hence from \eqref{longexpr} and \eqref{lastchain}, taking the limit $\gamma\to 0$, we obtain that
$$
\int_\Omega U(\gamma)^2\to\int_\Omega u_k(0)^2,\ \text{ as }\ \gamma\to0.
$$
Summing up, all this shows that
$$
\lambda_k(0)\ge \bar \lambda_k(1+o(1)),
$$
and since $\bar \lambda_k$ is an eigenvalue of problem \eqref{eq:neumann-squared}, this means that $(\bar \lambda_k,\bar u_k)=(\lambda_j(0),u_j(0))$ for some $1\le j\le k$. Supposing that $j<k$ leads to
\begin{equation*}\begin{split}
1&=\lim_{\gamma\to0}\int_\Omega u_j(0)u_k(\gamma)\\
&=\lim_{\gamma\to0}\int_\Omega (u_j(0)-u_j(\gamma))u_k(\gamma) +\lim_{\gamma\to0}\int_\Omega u_j(\gamma)u_k(\gamma)=0,
\end{split}\end{equation*}
a contradiction.
This concludes the proof in the case $\gamma\to0$.

The continuity of the eigenpair with respect to $\gamma\in[0,+\infty)$ can now be proven in the same way. As for limit $\gamma\to+\infty$, we can reverse the variable to $t=\gamma^{-1}$ and consider the eigenvalues
$$
\tilde \lambda_k(t)=t^{-1}\lambda_k(t^{-1})=
\min_{\substack{V\subseteq H^2_\ast(\Omega),\\ \dim V=k}}\max_{u\in V\setminus\{0\}} \,\frac{\int_\Omega |\Delta u|^2}{t\int_\Omega |u|^2 + \int_{\partial\Omega} |u|^2}.
$$
It is now clear that the same technique leads to the continuity of the eigenpair $(\tilde\lambda_k(t),t^{-\frac1 2 }u_k(t))$ with respect to $t$ up to $t=0$, which corresponds to the case $\gamma=+\infty$. Note that in this case the eigenfunctions must be multiplied by $t^{-\frac 1 2 }=\gamma^{\frac 1 2}$ in order to keep the normalization.
\end{proof}

\section{The eigenvalue problem in the ball}\label{section:ball}

In this section,  we set $\Omega = \mathbb{B}_1$ to be the unit ball in $\mathbb R^n$ and study the eigenvalue problem \eqref{eq: eigenvalue problem}. For simplicity, we will denote by  $\lambda$ the fourth-root of the eigenvalue, so the actual eigenvalue becomes $\lambda^4$.

A well known fact of polyharmonic problems on balls is that any solution of the equation
\begin{equation}
\label{biharmonic}
\Delta^2u=\lambda^4 u
\end{equation}
in $\mathbb{B}_1$ can be written as the product of a radial function times a spherical harmonic function (see e.g., \cite{Chasman}) using spherical coordinates $(r,\theta)$. Let us recall that the spherical harmonics are the eigenfunctions of the Laplace-Beltrami operator on the unit sphere $\mathbb S^{n-1}$, namely
\begin{equation}\label{Laplacian-spherical}-\Delta_{\mathbb{S}^{n-1}} \mathcal{Y}_{\ell,m}(\theta) = \ell(\ell+n-2)\mathcal{Y}_{\ell,m}(\theta).
\end{equation}
Here $\ell\in\mathbb N$, while $1\le m\le \mathcal C_{\ell,n}$, where $\mathcal C_{\ell,n}$ is the multiplicity of the eigenvalue $\ell(\ell+n-2)$. Any spherical harmonic associated with the eigenvalue $\ell(\ell+n-2)$ is said to be of order $\ell$, and we will drop the index $m$ whenever it is not necessary.

We rewrite equation \eqref{biharmonic} as
$$
(\Delta +\lambda^{2})(\Delta -\lambda^{2})u=0.
$$
For any $\lambda>0$, the general solution to $(\Delta +\lambda^{2})u=0$ has the form
$$
(A j_\ell(\lambda r)+B y_\ell(\lambda r))\mathcal{Y}_{\ell}(\theta),
$$
for some $\ell\in\mathbb N$, where  the functions $j_\ell$ and $y_\ell$ are the ultraspherical Bessel functions:
\begin{equation*}
    j_\ell(z) = z^{1-\frac n2} J_{\ell+\frac n2-1} (z),\quad y_\ell(z) = z^{1-\frac n2} Y_{\ell+\frac n 2 -1} (z).
\end{equation*}
We refer to e.g., \cite{olver2010nist}, which is the modern version of the classical \cite{Stegun}, for the definition and standard properties of Bessel functions. Notice that here any spherical harmonic of order $\ell$ produces a viable solution.

Similarly, for any $\lambda>0$ the general solution to $(\Delta -\lambda^{2})u=0$ has the form
$$
(C i_\ell(\lambda r)+D k_\ell(\lambda r))\mathcal{Y}_{\ell}(\theta),
$$
where the functions $i_\ell$ and $k_\ell$ are the ultraspherical modified Bessel functions:
\begin{equation*}
    i_\ell(z) = z^{1-\frac n2} I_{\ell+\frac n2-1} (z), \quad k_\ell(z) = z^{1-\frac n2} K_{\ell+\frac n 2 -1} (z).
\end{equation*}
Therefore, a general solution of \eqref{biharmonic} will take the form
$$
(A j_\ell(\lambda r)+B y_\ell(\lambda r)+C i_\ell(\lambda r)+D k_\ell(\lambda r))\mathcal{Y}_{\ell}(\theta).
$$
However, the functions $y_\ell$ and $k_\ell$ are singular at the origin and so is any linear combination (see \cite{Chasman})
$$
B y_\ell(\lambda r)+D k_\ell(\lambda r).
$$
Since eigenfunctions of \eqref{biharmonic} on the ball must be smooth at the origin, this means that $B=D=0$, so that such eigenfunctions will take the form
$$
(A j_\ell(\lambda r)+C i_\ell(\lambda r))\mathcal{Y}_{\ell}(\theta).
$$
The values of $A,C$, and $\lambda$ can now be obtained imposing the boundary conditions.

\begin{teor} \label{thm:general}
Fix $\gamma \geq 0$. The positve eigenvalues of
\begin{equation}\label{problem-ball}
    \begin{cases}
    \Delta^2 u = \lambda^4 u &\mbox{in } \mathbb B_1, \\
    \partial_\nu u = 0  &\mbox{on } \partial\mathbb B_1, \\
    \partial_\nu(\Delta u) = -\gamma \lambda^4 u  &\mbox{on } \partial\mathbb B_1,
    \end{cases}
\end{equation}
are all and only the solutions of the equation
\begin{equation} \label{gammaeq}
    2 j_\ell'(\lambda_\ell) i_\ell'(\lambda_\ell) + \gamma \lambda_\ell (j_\ell'(\lambda_\ell) i_\ell(\lambda_\ell) - i_\ell'(\lambda_\ell) j_\ell(\lambda_\ell)) = 0,
\end{equation}
for some $\ell\in\mathbb N$. For any such solution, the associated eigenfunctions take the form
$$
u_\ell(r)\mathcal Y_\ell(\theta)
$$
for some spherical harmonic $\mathcal Y_\ell$ of order $\ell$, where 
\begin{equation}\label{equation10}
    u_\ell(r) = i'_\ell(\lambda_\ell) j_\ell(\lambda_\ell r) - j'_\ell(\lambda_\ell) i_\ell(\lambda_\ell r).
\end{equation}
\end{teor}

\begin{proof} The proof is a minor variation of  \cite{Chasman}, but we detail it here for completeness.

As we saw at the beginning of the section, the function $u_\ell$ will be of the form
\begin{equation*} \label{solsball}
    u_\ell(r) = A_\ell j_\ell(\lambda r) + B_\ell i_\ell (\lambda r),
\end{equation*}
for some constants $A_\ell,B_\ell,\lambda$. Now we impose boundary conditions. Notice that the condition $\partial_\nu u=0$ readily implies
\begin{equation*}
    A_\ell j'_\ell(\lambda) + B_\ell i'_\ell(\lambda) = 0,
\end{equation*}
and from here we obtain \eqref{equation10}.

Moreover, we can consider the boundary conditions as a system of homogeneous equations
\begin{equation*}
    \begin{cases}
    \mathcal Mu= \partial_r u  = 0, \\ \mathcal Nu = \partial_r \Delta u + \lambda^4 \gamma u  = 0.
    \end{cases}
\end{equation*}
Since we search non-trivial eigenvalues, this system must have a non-trivial solution, so we are bound to impose that the following determinant vanishes at $r = 1$:
\begin{equation*}
\begin{vmatrix}
\mathcal M j_\ell & \mathcal M i_\ell \\ \mathcal N j_\ell & \mathcal N i_\ell
\end{vmatrix}.
\end{equation*}
Thus, any number $\lambda$ is an eigenvalue of problem \eqref{problem-ball} if and only if it satisfies the equation
\begin{equation*}
    -\lambda^4 [2 j_\ell'(\lambda) i_\ell'(\lambda) + \gamma \lambda (j_\ell'(\lambda) i_\ell(\lambda) - i_\ell'(\lambda) j_\ell(\lambda))] = 0.
\end{equation*}
Finally, if $\lambda \neq 0$, we get \eqref{gammaeq}. %This completes the proof of Theorem \ref{thm:general}.
\end{proof}

We remark that, in order to approximate numerically the value of the eigenvalues, one can rewrite \eqref{gammaeq} as a fixed-point problem for each $\ell$, for $\gamma\in(0,\infty)$:
\begin{equation}\label{fixed-point}
    \lambda = \frac{-2\gamma^{-1}}{\frac{i_\ell(\lambda)}{i'_\ell(\lambda)}-\frac{j_\ell(\lambda)}{j'_\ell(\lambda)}}
    =:\gamma^{-1}\Psi(\lambda).
\end{equation}
In Figure \ref{figure-Psi} below we plot the function $\Psi(\lambda)$  in dimensions $n=2,3$, for $\ell= 0,\ldots,5$, together with the identity function (in black).
\begin{figure}[H]
    \includegraphics[width = \textwidth]{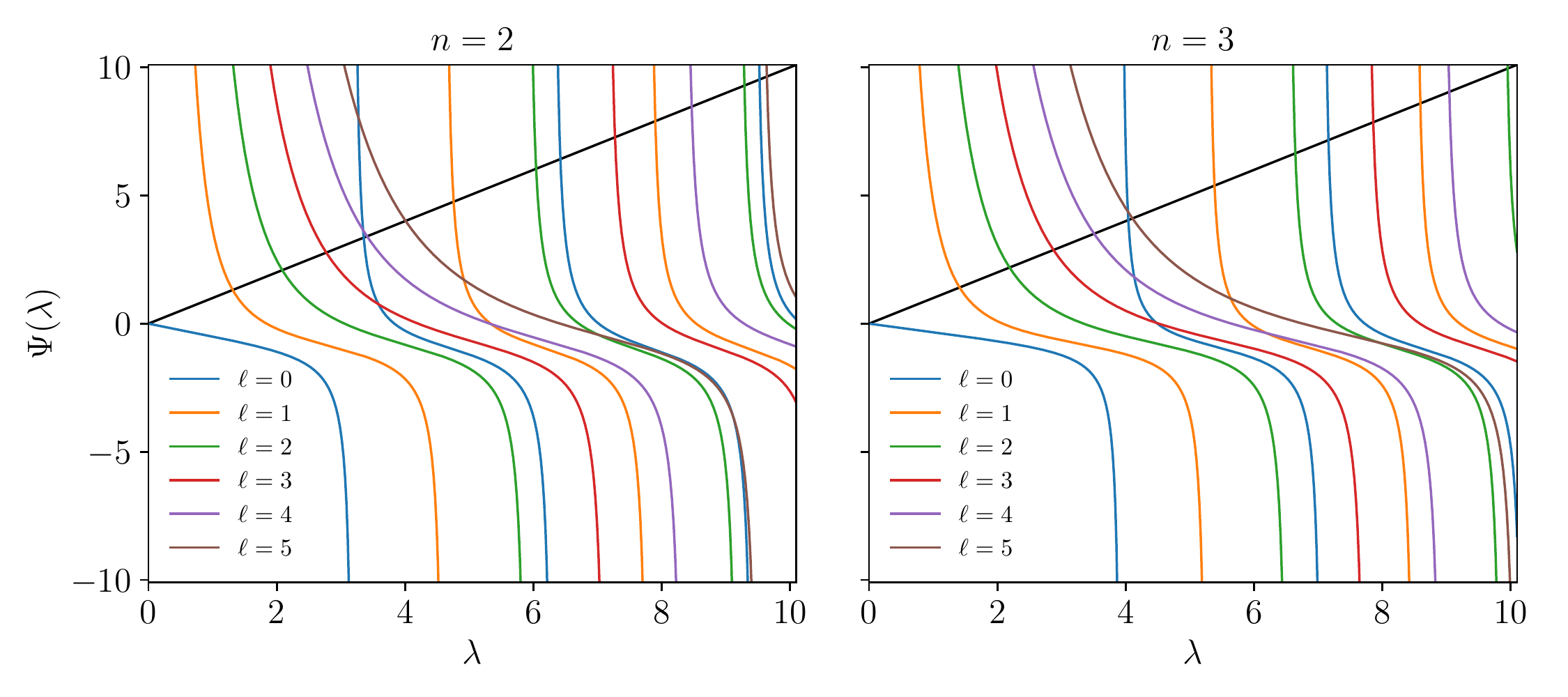}
    \caption{Function $\Psi(\lambda)$.}
    \label{figure-Psi}
\end{figure}

%Next, we are interested in studying the solution of equation \eqref{gammaeq} for different values of $\gamma$ and, in particular, showing that $\gamma=0$ and $\gamma=\infty$ are regular values for the eigenvalue problem \eqref{problem-ball}. Since all formulas for the eigenvalues are pretty much explicit, we do need to use abstract perturbation results.

%It is easy to observe that taking the limit $\gamma\to 0$ is quite trivial in the above calculation: just take into account the role of $\gamma$ in the fixed point argument  \eqref{fixed-point}, which reduces to $\gamma \lambda=\Psi(\lambda)$. The limit $\gamma\to\infty$, on the other hand, is more delicate, as we will see.

\begin{obs}\label{remark-lambda0}
For the case $\lambda=0$, solutions of problem $\Delta^2 u=0$  are, after projecting over spherical harmonics,  of the form
\begin{equation}\label{zero-eigenfunctions}
u_\ell(r) = \begin{cases} A+B\log r +Cr^2 + Dr^2\log r&\mbox{if }n =2,\,\ell=0,\\Ar + Br\log r+\frac{C}{r}+Dr^3&\mbox{if }n = 2,\,\ell=1,\\A +B\log r + Cr^2+\frac{D}{r^2}&\mbox{if }n = 4,\,\ell=0, \\Ar^\ell+Br^{-(\ell + n-2)}+Cr^{\ell + 2}+Dr^{-(\ell + n-4)}&\mbox{otherwise}, \end{cases}
\end{equation}
which shows that the eigenfunctions associated to the zero eigenvalue for \eqref{problem-ball} in the unit ball must be constant. 
\end{obs}

\subsection{The limit as $\gamma\to 0$}

As we already saw, the limit $\gamma\to 0$ is actually a regular value for problem \eqref{eq: eigenvalue problem}, so that from Theorem \ref{thm:general} we get

\begin{cor}
Consider the eigenvalue problem
\begin{equation} \label{eq:0 problem}
    \begin{cases}
    \Delta^2 u = \lambda^4 u & \mbox{in } \Omega, \\ \partial_\nu u = 0 & \mbox{on } \partial\Omega, \\ \partial_\nu(\Delta u) = 0 & \mbox{on } \partial\Omega.
    \end{cases}
\end{equation}
Then, any non-zero eigenvalue solves the equation
\begin{equation} \label{eq:jprime}
    j'_\ell(\lambda_\ell) = 0,
\end{equation}
for some $\ell$, and the corresponding eigenfunctions have the form
\begin{equation*}
    u_\ell(r) =  j_\ell(\lambda_\ell r).
\end{equation*}
\end{cor}
%\begin{proof}
%From equation \eqref{gammaeq} for $\gamma=0$, the desired result follows by just taking into account that the function $i_\ell(r)$ is always increasing for positive $r$, as the usual modified Bessel function of the first kind, so the function $i'_\ell(r)$ is positive for $r\geq 0$.
%\end{proof}

Note that here again the eigenvalues are the squares of those of the Neumann Laplacian, with the same associated eigenfunctions.

A more precise behavior of the tail eigenvalues for problem \eqref{eq:0 problem} is given in the following:

\begin{prop}\label{prop:behavior-0}
Fix $n>2$ and $\ell \in \mathbb{N}$. Let $(\lambda_k)_{k=1}^\infty$ be the set of positive zeroes of  equation \eqref{eq:jprime}.
Then, for any  $k$ large enough, we have:
\begin{equation*}
    0<c<\lambda_{k+1} - \lambda_k \leq 4\pi.
\end{equation*}
\end{prop}

\begin{proof}
Using the definition of $j_\ell(r)$, we can expand:
\begin{equation*}
    j'_\ell(r) = (r^{-s} J_{\ell+s}(r))' = r^{-s} J'_{\ell+s}(r) - s r^{-s-1} J_{\ell+s}(r).
\end{equation*}
Then, we can rewrite the equation \eqref{eq:jprime} as a fixed-point equation
\begin{equation*}
    r = \frac{sJ_{\ell+s}(r)}{J'_{\ell+s}(r)}.
\end{equation*}
Now, to study the function $\frac{s J_{\ell+s}(r)}{J'_{\ell+s}(r)}$ for $r$ big enough, we can use  Hankel's asymptotic expansion of the Bessel function   as $|r| \to \infty$ from \cite[page 364, section 9.2]{Stegun}
\begin{gather*}
    J_{\ell} (r) \approx \sqrt{\frac{2}{\pi r}} \left( P(\ell,r) \cos\left(r - \tfrac{2\ell-1}{4} \pi\right) - Q(\ell,r) \sin\left(r - \tfrac{2\ell-1}{4} \pi\right)\right),\\
    J^\prime_{\ell} (r) \approx -\sqrt{\frac{2}{\pi r}} \left( S(\ell,r) \cos\left(r - \tfrac{2\ell-1}{4} \pi\right) + R(\ell,r) \sin\left(r - \tfrac{2\ell-1}{4} \pi\right)\right),
\end{gather*}
where
\begin{equation*}
    \begin{split}
    &P(\ell,r) = \sum_{k=0}^\infty \frac{(-1)^k \langle\ell,2k\rangle}{(2r)^{2k}},\\
    &Q(\ell,r) = \sum_{k=0}^\infty \frac{(-1)^k \langle\ell,2k+1\rangle}{(2r)^{2k+1}},\\
    &R(\ell, z) = \sum_{k=0}^{\infty}(-1)^{k} \frac{4 \ell^{2}+16 k^{2}-1}{4 \ell^{2}-(4 k-1)^{2}} \frac{\langle\ell, 2 k\rangle}{(2 z)^{2k}},\\
    &S(\ell, z) = \sum_{k=0}^{\infty}(-1)^{k} \frac{4 \ell^{2}+4(2k +1)^{2}-1}{4 \ell^{2}-(4 k+1)^{2}} \frac{\langle\ell, 2 k+1\rangle}{(2 z)^{2k+1}},
\end{split}
\end{equation*}
for $\langle\ell,k\rangle = \prod\limits_{s=1}^{k} (4 \ell^2 - (2s+1)^2)$.

Then, by inspecting this expansions we get that, for every fixed $\ell$, we have that our function $\frac{s J_{\ell+s}(r)}{J'_{\ell+s}(r)}$ is asymptotically  a cotangent function. As both the cotangent and the identity are increasing functions, we easily obtain an estimate for the difference of two eigenvalues (see Figure \ref{fig:limit_cotangent}).

%they will cut once for each period of the function, so we would get an easy estimate for the distance between the eigenvalues ($2\pi$).
%In any case, we can take more terms in the functions $P,Q,R$ and $S$ and get estimates valid for smaller values of $r$. This functions are more complicated near $0$, but for large values the property of being increasing still holds. Taking just the next term, we get that main difference with the cotangent is that the singularities are not in the points where $\sin(r) = 0$ but on the points $\sin(r) = \frac{1}{r}$. This way, we have to extend the distance between eigenvalues to twice the period of $\sin(r)$ in view that this singularities may appear further than one period away from the previous one. So, we end up with the estimate that, for each fixed $\ell$ and large $r$, the eigenvalues are at most at $4\pi$ from the previous one.

\begin{figure}[H]
    \centering
    \includegraphics[width = 0.75\textwidth]{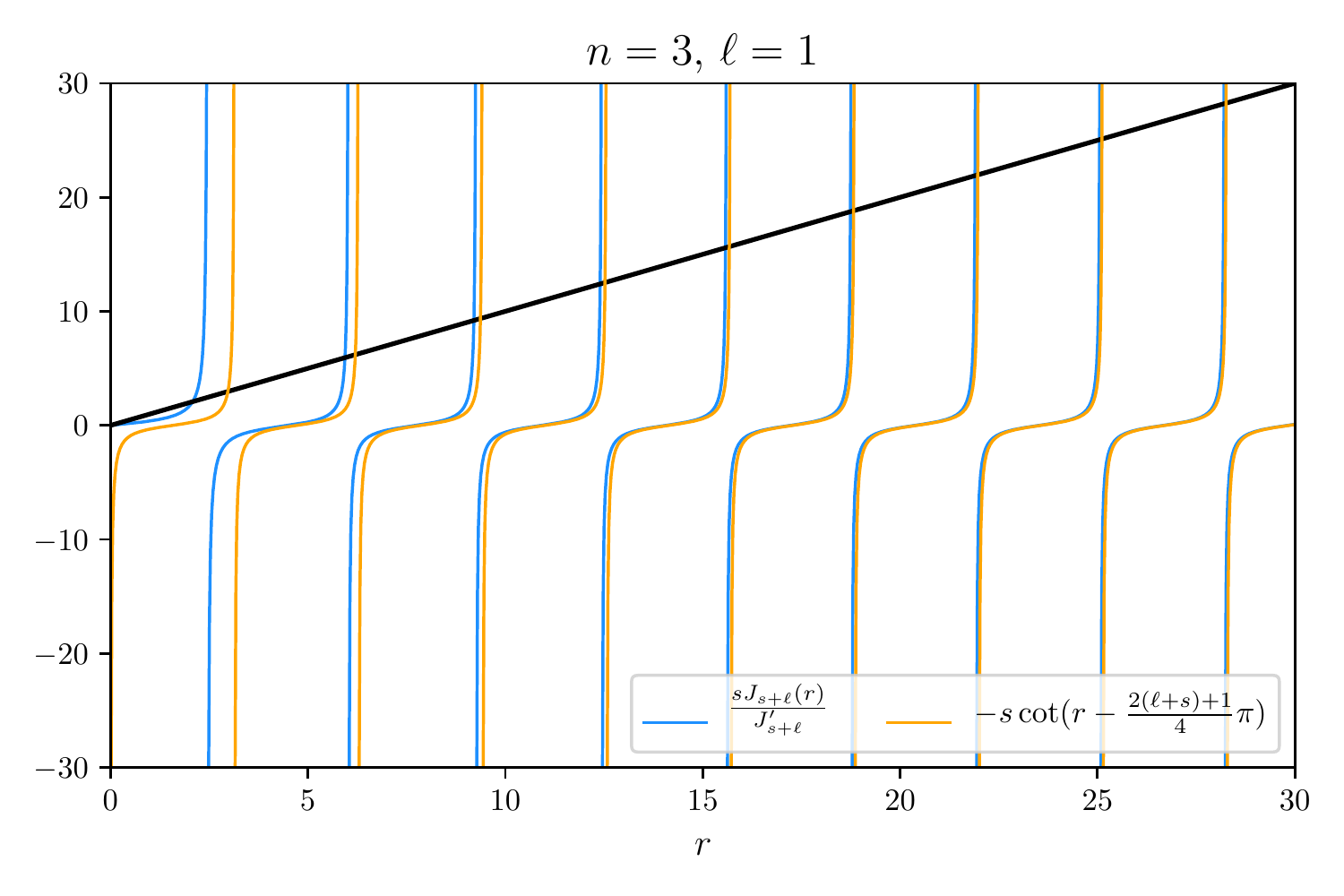}
    \caption{Comparison between function $\frac{s J_{\ell+s}(r)}{J'_{\ell+s}(r)}$ and its limit, a cotangent function scaled and displaced. The black line is the identity function.}
    \label{fig:limit_cotangent}
\end{figure}
\end{proof}

%On the other hand, we have to prove that this eigenvalues do not collapse. This is, there exists a constant $c>0$ so that $0<c\leq |\lambda_{k+1} - \lambda_k|$. But such a constant trivially exists as we are calculating the cutting points of two functions that only coincide once between singularities of one of them. Hence, between two any eigenvalues there is at least one singularity of the function $\dfrac{s J_{\ell+s}(r)}{J'_{\ell+s}(r)}$ and the constant would be twice the minimum of the distances from each eigenvalue to their nearest singularity.

\subsection{The limit as $\gamma\to \infty$}

In the case of the ball, the limiting problem as $\gamma\to\infty$ looks like
%We will see now that the case $\gamma=\infty$ is also a regular perturbation. First, we rewrite \eqref{problem-ball} as
%\begin{equation}\label{problem-ball1}
    %\left\{\begin{split}
    %&\Delta^2 u = \tfrac{1}{\gamma}\lambda^4 u  \mbox{ in } \Omega, \\
    %&\partial_\nu u = 0  \mbox{ on } \partial\Omega, \\
    %&\partial_\nu(\Delta u) = - \lambda^4 u  \mbox{ on } \partial\Omega.
    %\end{split}\right.
%\end{equation}
%Taking $\gamma = \infty$ in \eqref{problem-ball1} yields the following biharmonic Steklov eigenvalue problem
\begin{equation} \label{Steklov problem}
    \begin{cases}
    \Delta^2 u = 0 &\mbox{in } \Omega,\\
    \partial_\nu u= 0 &\mbox{on } \partial\Omega,\\
    \partial_\nu\Delta u = -\lambda^4 u &\mbox{on }\partial\Omega.
    \end{cases}
\end{equation}
%
%From the $\gamma$ equation \eqref{gammaeq} we could think that the eigenvalues satisfy:
%\begin{equation*}
%    j_\ell'(\lambda) i_\ell(\lambda) - i_\ell'(\lambda) j_\ell(\lambda)=0.
%\end{equation*}
%This would be exactly the radial part of our solution at $r=1$, so $u|_{\partial\Omega} = 0$.
%Then, the solution to the Steklov problem will be like \eqref{Steklov solutions} adding the boundary conditions. 
%
%Let us look at this limit equation first. It is straightforward to show (see Proposition \ref{prop steklov} in the Appendix) that:

We first recall the following result, which is a special case of the so-called Almansi decomposition for polyharmonic operators (see e.g., \cite{Almansi1,Almansi2}). For the sake of completeness we provide here a proof.

\begin{prop} \label{prop steklov}
Consider the problem
\begin{equation*}
    \Delta^2 u = 0 \mbox{ in } \mathbb{R}^n,
\end{equation*}
Then, any solution can be written as
\begin{equation} \label{Steklov solutions}
    u(r,\theta) = \sum\limits_{\ell=0}^\infty ( A_\ell r^\ell + C_\ell r^{\ell+2})\mathcal{Y}_\ell(\theta).
\end{equation}
\end{prop}

\begin{proof}
Since the Bilaplacian commutes with the Laplace-Beltrami operator $\Delta_{\mathbb S^{n-1}}$, we know we can employ a separation of variables technique, solving for each projection  $u_\ell$. For this, we write the Laplacian in spherical coordinates
\begin{equation*}
    \Delta u_\ell = \partial_{rr} u_\ell + \frac{n-1}{r} \partial_r u_\ell + \frac{-\ell(\ell+n-2)}{r^2} u_\ell.
\end{equation*}
Now, the change of variables $r = e^{-t}$ gives
\begin{equation} \label{eq:laplaciano}
    \Delta u_\ell = e^{2t} (\partial_t + \ell) (\partial_t - (\ell+n-2)) u_\ell,
\end{equation}
which yields the formula for the Bilaplacian
\begin{equation*}
    \Delta^2 u_\ell = e^{4t} (\partial_t + \ell) (\partial_t - (\ell+n-2))(\partial_t + (\ell+2)) (\partial_t - (\ell+n-4)) u_\ell.
\end{equation*}
Hence, our solutions will be of the form
\begin{equation*}
    u(t,\theta) = \sum\limits_{\ell=0}^\infty ( A_\ell e^{-\ell t} + B_\ell e^{(\ell+n-2)t} + C_\ell e^{-(\ell+2) t} + D_\ell e^{(\ell+n-4) t}) \mathcal{Y}_\ell(\theta).
\end{equation*}
Since eigenfunctions are regular at the origin (that is,  when $t \to \infty$), we must set $B_\ell=D_\ell=0$ and this yields the conclusion.
\end{proof}

From Proposition \ref{prop steklov} it is then straightforward to completely characterize eigenfunctions and eigenvalues of problem \eqref{Steklov problem} (see also \cite{Xia-Wang}).

\begin{prop} \label{Steklov sols}
For the  the biharmonic Steklov eigenvalue problem \eqref{Steklov problem} in
 $\Omega=\mathbb B_1$ the unit ball we know that  eigenvalues are 
$$\lambda^4 = \ell^2 (2\ell+n)$$ with eigenfunctions
$$u_\ell (r,\theta) = ((\ell +2) r^\ell -\ell r^{\ell+2})\mathcal{Y}_\ell(\theta).$$
\end{prop}

%We have to comment that, for simplicity in the notation, the eigenvalues $\lambda$ in Proposition \ref{Steklov sols} correspond to eigenvalues $\lambda^4$ in Theorem \ref{thm:general}.

Now we show convergence of the eigenvalues and the eigenfunctions. Even though this was known by Theorem \ref{thm-convergence} in the previous Section, here we obtain convergence directly by looking at the asymptotics. Indeed, the trick is to find the correct scale that allows to take the limit $\gamma \to \infty$. In the following, we work directly with the $\ell$-th projection.

\begin{prop}
Consider the eigenvectors $u_\ell$ in Theorem \ref{thm:general}. Then, the minimizers $\gamma^\frac{1-\ell}{2} u_\ell$ and their associated non-zero eigenvalues converge to those in Proposition \ref{Steklov sols} when $\gamma \to \infty$.
\end{prop}
\begin{proof}
By Theorem \ref{thm:general}, the radial coordinates of the minimizers of the energy $E_\gamma [u]$ are:
\begin{equation*}
    u_\ell(r) = i'_\ell(\gamma^{-\tfrac{1}{4}} \lambda) j_\ell(\gamma^{-\tfrac{1}{4}}\lambda r) - j'_\ell(\gamma^{-\tfrac{1}{4}}\lambda) i_\ell(\gamma^{-\tfrac{1}{4}}\lambda r),
\end{equation*}
and their respective eigenvalues satisfy the equation:
\begin{equation*}
    2 j_\ell'(\lambda) i_\ell'(\lambda) + \gamma \lambda (j_\ell'(\lambda) i_\ell(\lambda) - i_\ell'(\lambda) j_\ell(\lambda) = 0.
\end{equation*}
By using the series for Bessel ultraspherical functions in \cite{Chasman} (formulae (14) and (15)) for the eigenvalue equation we arrive to:
\begin{equation*}
    2\gamma^\frac{1-\ell}{2} (\lambda^{2\ell-2} \ell^2 a_0 - 2 \lambda^{2\ell-2} a_0 a_2) + O(\gamma^{-\frac{\ell}{2}}) = 0,
\end{equation*}
where the terms $a_k$ are the terms in the series and depend only on the dimension $n$ and $\ell$.
Thus, if we consider the scaled eigenvectors $\gamma^\frac{1-\ell}{2} u_\ell$, the equation reads
\begin{equation*}
    2(\lambda^{2\ell-2} \ell^2 a_0 - 2 \lambda^{2\ell-2} a_0 a_2) + O(\gamma^{-\frac{\ell}{2}}) = 0
\end{equation*}
Taking the limit $\gamma \to \infty$ and solving for $\lambda$ we arrive to
\begin{equation*}
    \lambda^4 = \frac{\ell^2 a_0}{2 a_2} = \ell^2 \frac{\Gamma(1+ \frac{n}{2}+\ell) 2^{2+\ell}}{2\Gamma(\frac{n}{2}+\ell)2^\ell} = \ell^2 (2n + \ell).
\end{equation*}

\end{proof}

\subsection{The fundamental mode in the ball}

We now aim to identify which eigenfunction corresponds to the fundamental mode in the ball or, more precisely, for what value of the index $\ell$ we obtain the fundamental mode. In particular, we will see it corresponds to $\ell=1$. We follow the ideas in \cite{Chasman} for a different Bilaplacian eigenvalue problem. Remark that, in contrast to the previous subsections, this smallest non-zero eigenvalue  will be denoted by $\mu^{(1)}$.

For each $\ell$, the seeked eigenvalues are the roots of $\Phi_\ell(z) = 0$ where $$\Phi_\ell(z) = 2j_\ell'(z)i_\ell'(z)  + \gamma z \big(j_\ell'(z)i_\ell(z) - i_\ell'(z)j_\ell(z)\big).$$
Analogously to the usual Bessel functions $J_\nu, I_\nu$, the ultraspherical Bessel functions have several useful recurrence relations \cite{Buoso-Chasman-Provenzano}. In particular we will use
\begin{equation}\label{jl_derivative_1}
    j_\ell'(z) = \frac{\ell}{z}j_\ell(z) -j_{\ell+1}(z),
\end{equation}
\begin{equation}\label{il_derivative_1}
    i_\ell'(z) = \frac{\ell}{z}i_\ell(z) +i_{\ell+1}(z),
\end{equation}
as well as
\begin{equation}\label{jl_derivative_2}
    j_\ell'(z) = j_{\ell-1}(z) - \frac{\ell + n - 2}{z}j_\ell(z).
\end{equation}

We will find the study of the zeros of the ultraspherical Bessel functions in \cite{Chasman:isoperimetric} particularly useful. Let $p_{1,1}$ be the smallest zero of $j_1'(z)$. Then, we have the following result.
\begin{lema}[Lemmas 5 and 6 in \cite{Chasman:isoperimetric}]\label{chasman0_j1_j1diff} We have $j_1(z) >0$ for $z\leq p_{1,1}$ and $j_1' > 0$ on $(0,p_{1,1})$.
\end{lema}
As a direct consequence we deduce:

\begin{lema}\label{lambda1_p11}
The smallest $\ell = 1$ eigenvalue on the unit ball, denoted by $\mu^{(1)}$, is smaller than $p_{1,1}$.
\end{lema}
\begin{proof}
$\mu^{(1)}$ is an eigenvalue with $\ell = 1$ if 
$\Phi_1(\mu^{(1)}) = 0$.
Now observe that 
$$\Phi_1(p_{1,1}) =- \gamma p_{1,1}i_1'(p_{1,1})j_1(p_{1,1}) < 0,$$
since $j_1(p_{1,1}) > 0$ and $i_1'$ is always positive.
Note also that $\Phi_1(0) = 2a_0^2$,
which means that $\Phi$ is positive near $z = 0$. Since $\mu^{(1)}$ is the smallest eigenvalue, it follows that $\mu^{(1)}<p_{1,1}$.
\end{proof}

\begin{lema}\label{j0_lemma}
We have $j_0 > 0$ on $(0,p_{1,1})$.
\end{lema}
\begin{proof}
Use \eqref{jl_derivative_2} to write $$j_0(z) = j_1'(z) + \frac{n-1}{z}j_1(z).$$
The claim follows now from Lemma \ref{chasman0_j1_j1diff}.
\end{proof}

We are now ready to prove the main result in this Subsection:

\begin{teor}\label{thm:fundamental-mode}
The smallest non-zero eigenvalue $\mu^{(1)}$ on the unit ball corresponds to $\ell = 1$.
\end{teor}
\begin{proof} We follow \cite{Chasman} and divide the proof in two stages. We show first that, for any fixed smooth radial function $R$ and $\ell \geq 1$, the energy $E_\gamma[R\mathcal{Y}_\ell]$ is minimized for $\ell = 1$. Then we show that, among $\ell = 1$ and $\ell = 0$, the smallest non-zero eigenvalue corresponds to $\ell = 1$.

The first part is immediate in our case. Indeed,
$$E_\gamma[R\mathcal{Y}_\ell] = \frac{\int_{\mathbb B_1}\Big(R''+\frac{n-1}{r}R' - \frac{R}{r^2}\ell(\ell+n-2)\Big)^2dr}{\int_{\mathbb B_1} R^2 r^{n-1}\,dr+\gamma |\mathbb S^{n-1}|R(1)^2},$$
and thus there is only one term in the numerator which depends on $\ell$. This term is increasing with $\ell$, and hence $E_\gamma[R\mathcal{Y}_\ell]$ reaches its minimum at $\ell = 1$.

We only need to look at the cases $\ell = 1$ and $\ell = 0$ now. Let $\mu^{(0)}$ be the smallest non-zero eigenvalue corresponding to $\ell = 0$. Due to Lemma \ref{lambda1_p11}, it suffices to show that $p_{1,1}<\mu^{(0)}$.

Clearly, for $\ell = 0$, one has $\Phi_0(0) = 0$. By using \eqref{jl_derivative_1} and \eqref{il_derivative_1} we can rewrite 
$$\Phi_0(z) = -2j_0(z)i_0(z) -\gamma z\big(j_1(z)i_0(z) + i_1(z)j_0(z)\big).$$
From Lemmas \ref{chasman0_j1_j1diff} and \ref{j0_lemma}, and by using the fact that $i_\ell$ and all its derivatives are always positive we find that $\Phi_0 < 0$ on $(0,p_{1,1})$. Hence $p_{1,1}<\mu^{(0)}$ and $\mu^{(1)}<\mu^{(0)}$, as claimed.
\end{proof}

\section{The eigenvalue problem in the annulus}\label{section:annulus}

In this section we aim to characterize the eigenfunctions and eigenvalues of our problem in the annulus $\Omega_a=\{x\in\mathbb{R}^n: a<|x|<1\}$ for each $0< a <1$. In particular we will consider 
\begin{equation}\label{problem-annulus}
    \begin{cases}
    \Delta^2 u = \lambda^4 u  &\mbox{in } \Omega_a, \\
    \partial_\nu u = 0  &\mbox{on } \partial\Omega_a, \\
    -\partial_\nu(\Delta u) = \gamma \lambda^4 u  &\mbox{on } \partial\Omega_a.
    \end{cases}
\end{equation}

We first look at  eigenfunctions correspoding to the zero-eigenvalue:

\begin{prop}
Eigenfunctions of \eqref{problem-annulus} for $\lambda=0$ are just the constants.
\end{prop}

\begin{proof}
Recalling Remark \ref{remark-lambda0}, we know that after projection over spherical harmonics, solutions to $\Delta^ 2 u = 0$  are of the form given by expression \eqref{zero-eigenfunctions}.
 Imposing boundary conditions in the first three cases yields the constants as the only possible solution. Assuming then that we are in the fourth case, direct calculations show that
$$\partial_r\Delta u_\ell = 2C\ell(2\ell+n)r^{\ell-1} + 2D(2\ell+n-4)(\ell+n-2)r^{-(\ell+n-1)}.$$
Thus imposing boundary conditions yields the system of equations
\begin{equation*}
\begin{pmatrix}
 \ell & 2-\ell - n  & \ell+2  & 4-\ell-n\\ 
 \ell a^{\ell-1}& (2-\ell - n)a^{-(\ell+n-1)}  & (\ell+2)a^{\ell+1}  & (4-\ell-n)a^{-(\ell + n-3)}\\
0&  0&  \ell(2\ell+n)&(2\ell+n-4)(\ell + n-2) \\
0 &  0 & \ell(2\ell+n)a^{\ell-1}&(2\ell+n-4)(\ell + n-2)a^{-(\ell+n-1)}
\end{pmatrix}
\begin{pmatrix}
A \\ B \\ C \\ D
\end{pmatrix} = \begin{pmatrix}
0 \\ 0 \\ 0 \\ 0
\end{pmatrix}.
\end{equation*}
In order to have non-trivial solutions, the matrix determinant needs to be equal to 0. But  this determinant reduces to
$$-(a^{\ell-1}-a^{-(\ell+n-1)})^2\ell^2(2\ell+n)(2\ell+n-4)(\ell+n-2)^2.$$
Note that we assumed that $2\ell + n\neq 4$ and $\ell + n > 2$. Hence the only possibility for this determinant to be zero is setting $\ell = 0$ when $n\neq 2,4$. But, in this case, the rank of this matrix is three, so the eigenspace is still one-dimensional,
as desired.
\end{proof}

\begin{teor}
Non-zero eigenvalues of problem \eqref{problem-annulus} are given by solutions of the equation
\begin{equation}\label{det-W}
    \det W_\ell(\lambda) = 0,
\end{equation}
where $W_\ell$ is the matrix given in \eqref{big-matrix}, for each $\ell=0,1,\ldots$
\end{teor}

\begin{proof}
We will follow the same steps as in the proof of Theorem \ref{thm:general}, for $\lambda\neq 0$. More precisely, after projection over spherical harmonics, we use ultraspherical Bessel functions to write the radial coordinate of the eigenfunction
\begin{equation}\label{eigenfunction-l}
    u_\ell (r) = A_\ell j_\ell(\lambda r) + B_\ell y_\ell(\lambda r) + C_\ell i_\ell(\lambda r) + D_\ell k_\ell(\lambda r),
\end{equation}
for some constants $A,B,C,D$, which will be fixed from the boundary conditions. Imposing the zero-Neumann and the third order conditions at both boundaries yields the system of equations, in matrix form,
%
\begin{comment}
First, the zero-Neumann condition yields the first two restrictions:
\begin{equation*}\begin{split}
    &A_\ell j^\prime_\ell(\lambda a) + B_\ell y^\prime_\ell(\lambda a) + C_\ell i^\prime_\ell(\lambda a) + D_\ell k^\prime_\ell(\lambda a)= 0, \\
    &A_\ell j^\prime_\ell(\lambda ) + B_\ell y^\prime_\ell(\lambda ) + C_\ell i^\prime_\ell(\lambda ) + D_\ell k^\prime_\ell(\lambda )= 0.
\end{split}\end{equation*}
%
Analogously, the third-order boundary condition gives us two more equations:
\begin{equation*}\begin{split}
    &A_\ell j^\prime_\ell(\lambda a) + B_\ell y^\prime_\ell(\lambda a) - C_\ell i^\prime_\ell(\lambda a) - D_\ell k^\prime_\ell(\lambda a)\\&\qquad= \gamma\lambda\big(A_\ell j_\ell(\lambda a) + B_\ell y_\ell(\lambda a) + C_\ell i_\ell(\lambda a) + D_\ell k_\ell(\lambda a)\big),\\
    &A_\ell j^\prime_\ell(\lambda ) + B_\ell y^\prime_\ell(\lambda ) - C_\ell i^\prime_\ell(\lambda ) - D_\ell k_\ell(\lambda )\\&\qquad= -\gamma\lambda\big(A_\ell j_\ell(\lambda ) + B_\ell y_\ell(\lambda ) + C_\ell i_\ell(\lambda ) + D_\ell k_\ell(\lambda )\big).
\end{split}\end{equation*}
These four can be rewritten in matrix form as
\end{comment}
%
\begin{equation*}
    W_\ell(\lambda) \cdot
\begin{pmatrix}
A \\ B \\ C \\ D
\end{pmatrix} = \begin{pmatrix}
0 \\ 0 \\ 0 \\ 0
\end{pmatrix} 
\end{equation*}
where we have used $W_\ell (\lambda)$ to denote the matrix
\begin{equation}\label{big-matrix} \raggedleft
\begin{pmatrix}
j^\prime_\ell(\lambda a) &  y^\prime_\ell(\lambda a) &  i^\prime_\ell(\lambda a) & k^\prime_\ell(\lambda a) \\ 
j^\prime_\ell(\lambda ) &  y^\prime_\ell(\lambda ) &  i^\prime_\ell(\lambda ) & k^\prime_\ell(\lambda ) \\
j^\prime_\ell(\lambda a) - \gamma\lambda j_\ell(\lambda a) &  y^\prime_\ell(\lambda a) - \gamma\lambda y_\ell(\lambda a) &  -i^\prime_\ell(\lambda a) -\gamma\lambda i_\ell(\lambda a)& -k^\prime_\ell(\lambda a) -\gamma\lambda k_\ell(\lambda a)\\
j^\prime_\ell(\lambda ) + \gamma\lambda j_\ell(\lambda ) &  y^\prime_\ell(\lambda ) + \gamma\lambda y_\ell(\lambda ) &  -i^\prime_\ell(\lambda ) +\gamma\lambda i_\ell(\lambda )& -k^\prime_\ell(\lambda a) +\gamma\lambda k_\ell(\lambda )\\
\end{pmatrix}
\end{equation}
Non-trivial solutions exist when its determinant vanishes, which is precisely condition \eqref{det-W}. This completes the proof of the Theorem.\\
\end{proof}

Now we would like to study the behavior of the eigenvalues  by looking at the solutions of equation \eqref{det-W} depending on the size of the annulus (the parameter $a$), for each $\gamma$. 
In order to illustrate the general picture, let us start with dimension   $n =2$.
In this case, $s = 0$ and the ultraspherical Bessel functions become the usual ones. Hence the matrix $W_\ell(\lambda)$ reduces to
\begin{equation*}
\begin{pmatrix}
J^\prime_\ell(\lambda a) &  Y^\prime_\ell(\lambda a) &  I^\prime_\ell(\lambda a) & K^\prime_\ell(\lambda a) \\
J^\prime_\ell(\lambda) &  Y^\prime_\ell(\lambda) &  I^\prime_\ell(\lambda) & K^\prime_\ell(\lambda) \\
J^\prime_\ell(\lambda a) - \gamma\lambda J_\ell(\lambda a) &  Y^\prime_\ell(\lambda a) - \gamma\lambda Y_\ell(\lambda a) &  -I^\prime_\ell(\lambda a) -\gamma\lambda I_\ell(\lambda a)& -K^\prime_\ell(\lambda a) -\gamma\lambda K_\ell(\lambda a)\\
J^\prime_\ell(\lambda) + \gamma\lambda J_\ell(\lambda) &  Y^\prime_\ell(\lambda) +\gamma\lambda Y_\ell(\lambda) &  -I^\prime_\ell(\lambda) +\gamma\lambda I_\ell(\lambda)& -K^\prime_\ell(\lambda) +\gamma\lambda K_\ell(\lambda)\\
\end{pmatrix}
\end{equation*}
Imposing equation \eqref{det-W} yields Figure \ref{n2_annulus} below for the eigenvalues. \begin{figure}[H]
    \centering
    \includegraphics[width = \textwidth]{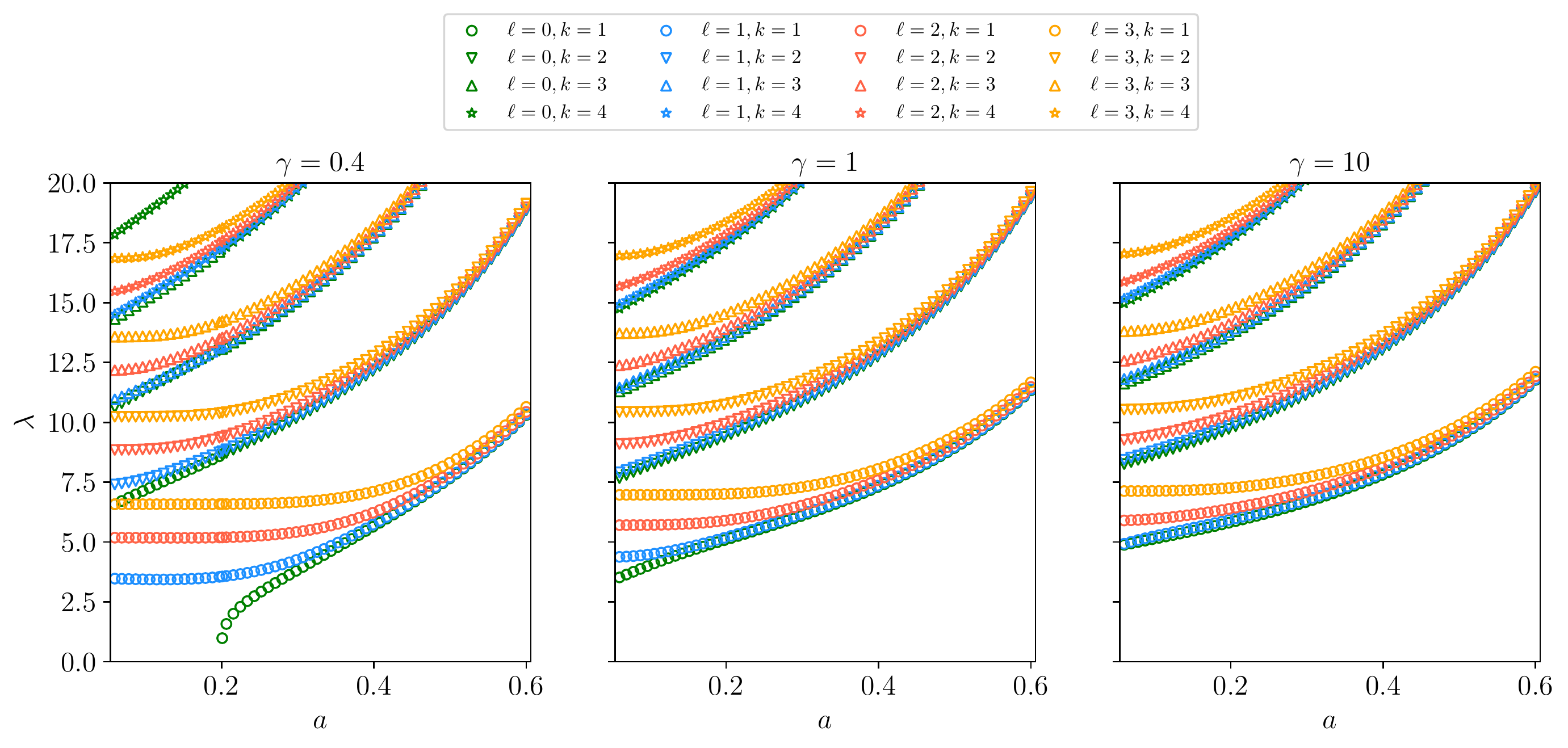}
    \caption{First four eigenvalues ($k = 1,2,3,4$) for $\ell = 0,1,2,3$ and different values of $\gamma$ and $a$ (in dimension $n = 2$).}
    \label{n2_annulus}
\end{figure}
Note that for small values of $\gamma$ ($\gamma = 0.4$ in Figure \ref{n2_annulus}) a bifurcation from $\lambda = 0$ appears for a specific value of $a$. To understand why this happens, we also plot in Figure \ref{n2_annulus_det} $\det W_\ell(\lambda)$ for different values of $\ell$ and observe that in the $\ell = 0$ case, $\det W_0(0)$ changes sign in a neighborhood of $a \approx 0.2$. Whenever $\det W_0(0)$ is positive, there exists a solution for $\ell = 0$ with a non-zero eigenvalue which is smaller than the smallest $\ell = 1$ eigenvalue. However, when $\det W_0(0)$ is negative, the smallest non-zero eigenvalue corresponds to $\ell = 1$. 
\begin{figure}[ht]
    \centering
    \includegraphics[width = 0.85\textwidth]{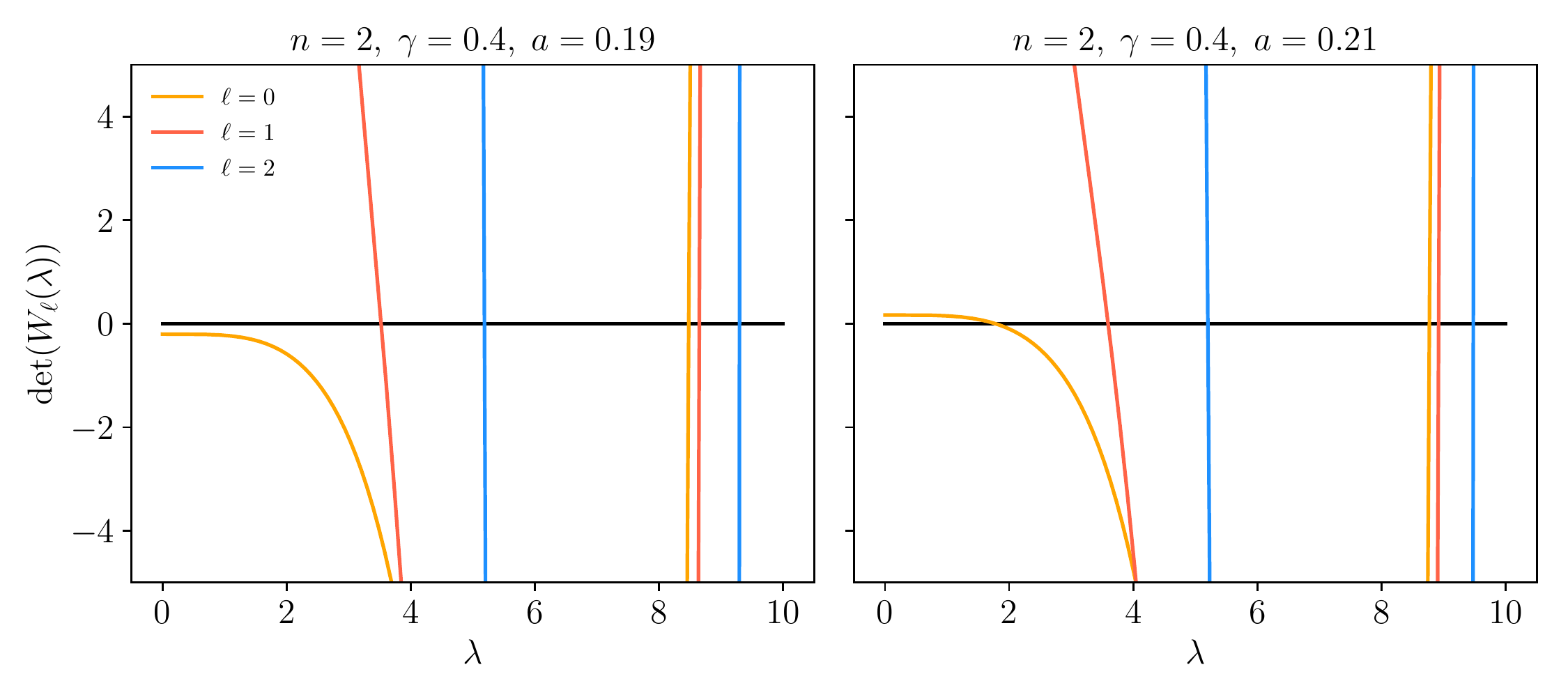}
    \caption{$\det W_\ell(\lambda)$ as a function of $\lambda$ for $n = 2, \gamma =0.4,\ell = 0,1,2$ for the cases $a = 0.19$ and $a=0.21$.}
    \label{n2_annulus_det}
\end{figure}

According to this discussion, finding when the determinant is equal to zero should tell us exactly when the bifurcation happens. To see this, 
use the asymptotic expansions of the Bessel functions expansions for $\ell = 0$ to  obtain that 
$$\lim_{\lambda\rightarrow 0^+} \det W_0(\lambda) = \frac{2(a+1)^2(a-1)(a+2\gamma - 1)}{\pi a^2}$$
and, hence, the bifurcation should happen when $a +2\gamma - 1 = 0$.  For example, when $\gamma = 0.4$ the bifurcation occurs at $a = 0.2$ which agrees with the plot. This bifurcation exists as long as $0<1-2\gamma < 1$. 

In the case $n\geq3$ we can follow a similar procedure. Using the expansions for the ultraspherical Bessel functions we have instead
\begin{align*}\det W_0(\lambda) = \frac{2\pi s^2(a^{2s+2}-1)\big(a^{2s+2}+2\gamma (s+1) a^{2s+1}+2\gamma(s+1)-1\big)}{a^{4s+2}(s+1)^2\Gamma(1+s)^2\Gamma(1-s)^2\sin^2(\pi s)}\lambda^{-4s}+o(\lambda)\\=\frac{2(a^{2s+2}-1)\big(a^{2s+2}+2\gamma (s+1) a^{2s+1}+2\gamma(s+1)-1\big)}{\pi a^{4s+2}(s+1)^2}\lambda^{-4s}+o(\lambda),\end{align*}
 By a similar argument, the bifurcation appears when $$a^{2s+2}+2\gamma (s+1) a^{2s+1}+2\gamma(s+1)-1=0$$
or, recalling the definition of $s$, when
$$F(a):=\frac{1-a^{n}}{1+a^{n-1}}=\gamma n.$$
Since the function $F(a)$ is strictly decreasing for $a\in(0,1)$, $F(0) = 1$ and $\lim_{a\rightarrow 1^-}F(a) = 0$,  this equation has a unique solution if and only if $\gamma n < 1$. Hence there exists a bifurcation as long as $0<\gamma < \frac{1}{n}$ (See Figure \ref{n4_annulus} for $\gamma = 0.2$). Note, finally, that in the limit $s\rightarrow 0$ we recover the previous results on dimension $n=2$.

\begin{figure}[H]
    \centering
    \includegraphics[width = \textwidth]{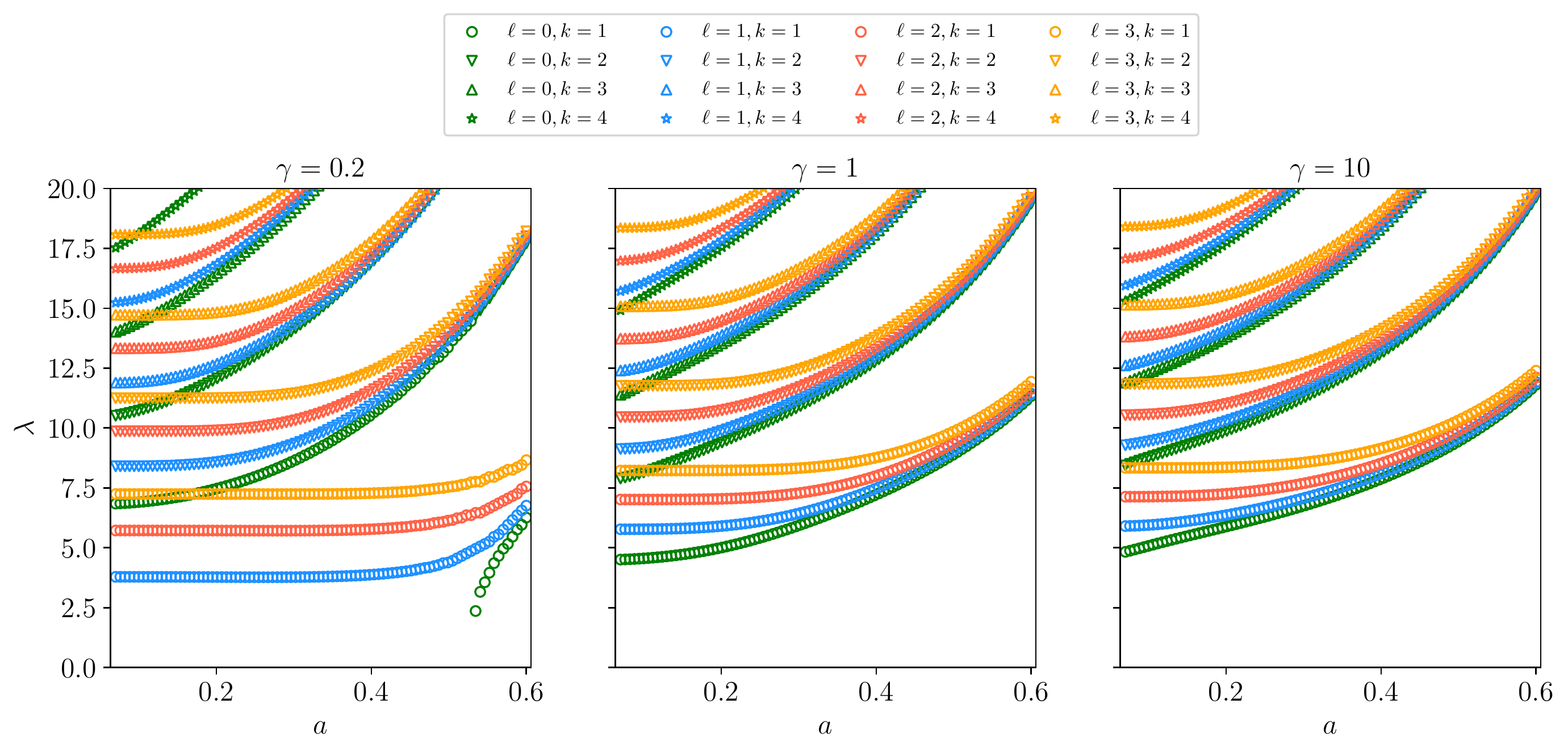}
    \caption{First four eigenvalues ($k = 1,2,3,4$) for $\ell = 0,1,2,3$ and different values of $\gamma$ and $a$ ($n = 4$).}
    \label{n4_annulus}
\end{figure}

We summarize some of our findings  below, as seen in Figure \ref{n2_annulus} (for dimension $n=2$) and Figure \ref{n4_annulus} (for $n=4$):

\begin{prop} Assume that $n\geq 2$ and $\gamma\in(0,\tfrac{1}{n})$. There exists $a_*\in(0,1)$ such that,
or $a_*<a<1$, the eigenfunction corresponding to the lowest non-zero eigenvalue is radially symmetric, while if $0<a<a_*$, the  eigenfunction is not radially symmetric but corresponds to the mode $\ell=1$. For $n=2$, $a_*=1-2\gamma$.
\end{prop}

We remark that a similar phenomena can be  observed in the calculation of the lowest non-zero Steklov eigenvalue in a 2-dimensional annulus from \cite{Fraser-Schoen}. Indeed, there is a critical size that, above which, the eigenfunction is radially symmetric, while below this critical value, the eigenfunction corresponds to the mode $\ell=1$. Moreover, this bifurcation is present in some four-dimensional  problems for the Paneitz operator and its associated third-order boundary operator \cite{Gonzalez-Saez} in the context of conformal geometry and classification of locally conformally flat manifolds.

\subsection{The punctured ball}

While for $0<a<1$ eigenfunctions and eigenvalues on the annulus $\Omega_a$ can be treated as we saw in the previous section, the case of the punctured ball $\Omega_0=\mathbb{B}_1(\mathbb{R}^n)\setminus\{0\}$ has to be handled in a different way. We remark that this case is interesting by its own both as a limiting case for $a\to 0$, and as a pathological case since the boundary of $\Omega_0$ is no longer smooth. In particular, for this reason problem \eqref{problem-annulus} on $\Omega_0$ has to be considered in the sense of the weak formulation \eqref{eq:variationeq} (see also \cite{buosoparini} and the references therein for similar considerations). 

First of all we observe that, thanks to the Removable Singularities Theorem (see e.g., \cite[Section 1.2.5]{mazpob}) we have the identification
$$
H^2(\Omega_0)=H^2(\mathbb{B}_1(\mathbb{R}^n)),
$$
in the sense that the standard embedding $H^2(\mathbb{B}_1(\mathbb{R}^n))\hookrightarrow H^2(\Omega_0)$ is surjective (see also \cite[Section 2.6]{ziemer} for more information on removable singularities for Sobolev spaces). Nevertheless, in order to properly identify the space $H^2_*(\Omega_0)$, we must analyze the behavior of traces of $H^2$-functions on $\partial\Omega_0$. More specifically we need to understand what it means to trace $\nabla u$ at the origin. However, we observe that $\nabla u\in H^1(\mathbb B_1)$ (by the Removable Singularities Theorem), hence it is not possible to trace the gradient on the origin for any $n\ge 2$ (cf.\ \cite[Section 5.2]{burenkov}). It is still possible though to trace the gradient on the smooth part of $\partial\Omega_0$, namely $\partial\mathbb B_1$. Combining these observations with the definition \eqref{h2star} we deduce that
 $$
 H^2_*(\Omega_0)=H^2_*(\mathbb{B}_1).
 $$

Finally, we observe that on the smooth part of $\partial\Omega_0$, namely $\partial\mathbb B_1$, the boundary conditions look as expected:
$$
\partial_\nu u = 0, \quad \text{ and } \quad -\partial_\nu(\Delta u) = \gamma \lambda^4 u \quad \text{ on } \partial\mathbb B_1.
$$

All these considerations can be summed up in the following

\begin{teor}
For any $n\ge 2$, the eigenvalues and the eigenfunctions of problem \eqref{problem-annulus} on the punctured ball $\Omega_0$ coincide with those of problem 
\eqref{eq: eigenvalue problem} on the ball $\mathbb B_1(\mathbb R^n)$. In particular, the eigenfunctions on the punctured ball can be analytically continued to the whole ball. 
\end{teor}

Let us conclude by observing that, in other types of eigenvalue problems for the Bilaplacian, this behaviour on the punctured ball may or may not appear depending on the boundary conditions: we refer for instance to \cite{buosoparini, coffman1, coffman2} where the case $n=2$ is considered and the eigenfunctions and the eigenvalues on the punctured disk are different from those of the unit disk. In fact, while it is not possible to trace the gradient of an $H^2$-function on a singleton, it is still possible to trace the function itself, so that
$$
H^2_0(\Omega_0)\subsetneqq H^2_0(\mathbb B_1(\mathbb R^n))
$$
for $n=2,3$, showing that the Dirichlet Bilaplacian presents a different behaviour in the disk and in the punctured disk.

\bigskip

\textbf{Acknowledgements:} D.\ Buoso is a member of the Gruppo Nazionale per l'Analisi
Matematica, la Probabilit\`a e le loro Applicazioni (GNAMPA) of the Istituto Naziona\-le di Alta Matematica (INdAM). This work was initiated while C. Falc\'o was participating in the program ``Introducci\'on a la Investigaci\'on Severo Ochoa" grant at the Instituto de Ciencias Matem\'aticas (ICMAT).
M.d.M. Gonz\'alez  acknowledges financial support from the Spanish  Government, grant numbers MTM2017-85757-P, PID2020-113596GB-I00; additionally, Grant RED2018-102650-T funded by MCIN/AEI/10.13039/501100011033, and the ``Severo Ochoa Programme for Centers of Excellence in R\&D'' (CEX2019-000904-S). M. Miranda completed this work thanks to the contract ``Ayuda extraordinaria a Centros de Excelencia Severo Ochoa" at the Instituto de Ciencias Matem\'aticas (ICMAT).

\bigskip
\end{document}